\documentclass[11pt]{article}
\usepackage{eurosym}
\usepackage{amsmath,amssymb,amsthm,amsfonts}
\usepackage{bbm}
\usepackage{subcaption}
\usepackage{appendix}
\usepackage{amsmath}
\usepackage{mathrsfs}
\usepackage{amsfonts}
\usepackage{amssymb}
\usepackage{multirow}
\usepackage{rotating}
\usepackage{xcolor}
\usepackage{lscape}
\usepackage{verbatim}
\usepackage{graphicx}
\usepackage{cancel}
\usepackage[normalem]{ulem}
\usepackage{url}
\usepackage{booktabs}
\usepackage{authblk}

\definecolor{Alecolor}{rgb}{0,0.5,0.5}

\usepackage{algorithm}
\usepackage{algpseudocode}

\numberwithin{equation}{section} 

\newtheorem{theorem}{Theorem}[section]
\newtheorem{corollary}{Corollary}[section]
\newtheorem{lemma}{Lemma}[section]
\newtheorem{proposition}{Proposition}[section]
\newtheorem{definition}{Definition}[section]
\newtheorem{remark}{Remark}
\newtheorem{example}{Example}
 
\usepackage{tikz}

\def\xx{\boldsymbol{x}}

\def\ee{\boldsymbol{e}}

\def\aa{\boldsymbol{a}}

\def\KK{\boldsymbol{K}}

\def\RRR{\boldsymbol{R}}
\def\PP{\mathbb{P}}

\def\ww{\boldsymbol{w}}

\def\ii{\boldsymbol{i}}

\def\II{\boldsymbol{I}}

\def\HHH{\mathcal{H}}

\def\EEE{\mathcal{E}}

\def\PPP{\mathcal{P}}

\def\AAA{\mathcal{A}}

\def\pp{\boldsymbol{p}}
\def\ppi{\boldsymbol{\pi}}

\def\kk{\boldsymbol{k}}

\def\RR{\mathbb{R}}

\def\BBB{\mathcal{B}}
\def\DDD{\mathcal{D}}

\def\zz{\boldsymbol{z}}

\def\JJ{\boldsymbol{J}}

\def\DDD{\mathcal{D}_d}

\usepackage{tikz}

\usepackage{enumitem}
\newcommand{\subscript}[2]{$#1 _ #2$}

\begin{document}

	\title{Extremal negative dependence and the strongly Rayleigh property}
	\author[1]{Hélène Cossette}
	\author[1]{Etienne Marceau}
	\author[2]{Alessandro Mutti}
	\author[2]{Patrizia Semeraro}
	
	\affil[1]{\textit{\'Ecole d'actuariat, Université Laval}}
	\affil[2]{\textit{Department of Mathematical Sciences, Politecnico di Torino}}
	
	\maketitle

	\begin{abstract}
		We provide a geometrical characterization of extremal negative dependence as a convex polytope in the simplex of multidimensional Bernoulli distributions, and we prove that it is an antichain that satisfies some minimality conditions with respect to the strongest negative dependence orders.  
		We study the strongly Rayleigh property within this class and explicitly find a distribution that satisfies this property by maximizing the entropy. 
		Furthermore, we construct a chain for the supermodular order starting from extremal negative dependence to independence by mixing the maximum entropy strongly Rayleigh distribution with independence.
	\end{abstract}

	
	\noindent \textbf{MSC2020 subject classification}: Primary 60E05, 60E15; secondary 62H05. 
	
	\section{Introduction}
	
	Negative dependence properties of multidimensional Bernoulli distributions are important in many areas of probability (e.g. \cite{boutsikas2000bound}, \cite{hermon2023modified}), statistics (e.g. \cite{branden2012negative}, \cite{gerber2019negative}), combinatorics, stochastic processes, statistical physics \cite{duminil2016quantitative} as well as in many applications, see \cite{borcea2009negative}, \cite{pemantle2000towards}, and references therein. 
	However,  the theory of negative dependence is more challenging than that of positive dependence, and there are still open problems. A key concept that is agreed upon for positive dependence and gives rise to different definitions for negative dependence is extremal dependence.
	This problem is useful to solve many optimization problems,  as discussed in \cite{puccetti2015extremal} from different fields, such as  applied mathematics \cite{lee2014multidimensional},   mathematical finance \cite{bernard2017robust}, \cite{fontana2021model}, and actuarial science \cite{cheung2013characterization}, \cite{lauzier2023pairwise}. 
	
	The strength of  dependence is compared through dependence orders $\preceq_*$, which are partial orders defined on Fréchet classes of probability distributions with the same one-dimensional marginals, satisfying some conditions (see \cite{joe1997multivariate}, \cite{kimeldorf1978monotone}, \cite{kimeldorf1989framework}, and \cite{tchen1980inequalities} as standard references).
	In each Fréchet class, dependence orders have an upper bound, the upper Fréchet bound, and a lower bound, the lower Fréchet bound. 
	The upper Fréchet bound is the joint distribution of a vector with components that are increasing transformations of a common random variable and it is, therefore, usually named the distribution of comonotonic random variables \cite{kimeldorf1978monotone}; it corresponds to the extremal positive dependence. 
	Extremal negative dependence is an intuitive concept in dimension two. A two-dimensional random vector is countermonotonic if its components are oppositely ordered, and in this case, its distribution is the lower Fréchet bound. In dimensions higher than two, the lower Fréchet bound is not always a distribution function, see \cite{dall2012advances}.
	Therefore, the generalization of countermonotonicity to higher dimensions is still an open issue, although various notions are proposed in the literature. 
	In \cite{puccetti2015extremal}, the authors review the most popular extremal dependence properties and they introduce a new one: $\Sigma$-countermonotonicity, the notion of extremal dependence we focus on.
	
	Whenever the lower Fréchet bound belongs to a Fréchet class, it is the unique $\Sigma$-countermonotonic distribution within the class. 
	Otherwise, when the lower Fréchet bound is not attainable, there are many $\Sigma$-countermonotonic vectors, and it is more difficult to interpret their role in the theory of negative dependence. 
	
	In this paper, we characterize extremal negative dependence in the class of multidimensional Bernoulli variables by completing the theory of negative dependence in \cite{borcea2009negative} with the theory of dependence orders  (see \cite{joe1997multivariate} and \cite{shaked2007stochastic}). Specifically, we prove that in the class of multidimensional Bernoulli distributions, extremal negative dependence is properly represented by the notion of $\Sigma$-countermonotonicity introduced in \cite{puccetti2015extremal}. 
	To this aim, we geometrically characterize this class, we study the main negative dependence properties across this class, and we prove that it satisfies a weak minimality condition with respect to the main dependence orders.
	
	Formally, let $\mathcal{B}_d$ be the class of $d$-dimensional Bernoulli probability mass functions $f_{\II}: \{0,1\}^d \rightarrow [0,1]$ and let $\BBB_d(\pp)= \BBB_d(p_1,\ldots, p_d)$ be the Fréchet class of multidimensional Bernoulli distributions with Bernoulli marginal distributions with means $p_j$, $j \in \{1,\ldots, d\}$.
	Negative dependence properties can be defined as subsets of $\BBB_d(\pp)$, see \cite{kimeldorf1989framework}, and dependence orders are partial orders $\preceq_*$ on $\BBB_d(\pp)$, allowing to compare the strength of a negative dependence property. 
	We focus on three negative dependence properties: negative association (\cite{joag1983negative} and \cite{shaked2007stochastic}), the weaker negative supermodular property (\cite{hu2000negatively}), and strongly Rayleigh property (as a negative dependence property) that has been introduced in \cite{borcea2009negative} for the Bernoulli variables using geometrical properties of the zero sets of probability generating polynomials. 
	The strongly Rayleigh property is the strongest negative dependence notion, since it implies negative association. 
	However, the strength of the strongly Rayleigh property cannot be compared in the Fréchet class $\mathcal{B}_d(\pp)$. Indeed, $\mathcal{B}_d(\pp)$ is an antichain of the partial order introduced in \cite{borcea2009negative} to compare the strength of the strongly Rayleigh property in the class $\mathcal{B}_d$.
	This order, indeed, is not a dependence order according to the classical definition \cite{kimeldorf1989framework}.
	We compare the strength of dependence using the dependence orders associated to negative association, i.e., weak association order $\preceq_{w-assoc}$, and to the negative supermodular property, that is the supermodular order $\preceq_{sm}$. It is known that $\preceq_{w-assoc}$ implies $\preceq_{sm}$, see \cite{shaked2007stochastic}.
	
	Our first significant result is the geometrical characterization of $\Sigma$-countermonotonicity in Section~\ref{sec:characterization}, where we prove that the class $\mathcal{B}_d^{\Sigma}(\pp)$ of $\Sigma$-countermonotonic distributions in $\mathcal{B}_d(\pp)$ is a convex polytope and we provide its extremal points.  
	
	We study the strongly Rayleigh property in the class $\mathcal{B}_d^{\Sigma}(\pp)$. 
	We provide counterexamples that show that $\Sigma$-countermonotonicity does not imply the strongly Rayleigh property and that it is not implied by it. 
	Our second important result,  Theorem \ref{thm:maxentr}, is to prove that there exists at least one strongly Rayleigh distribution in $\mathcal{B}_d^{\Sigma}(\pp)$, and that it maximizes the entropy in the class.
	We also provide a construction to find another strongly Rayleigh pmf in the class $\mathcal{B}_d^{\Sigma}(\pp)$, showing that in general the maximum entropy pmf in $\BBB_d^{\Sigma}(\pp)$ is not the unique pmf in $\BBB_d^{\Sigma}(\pp)$ with the strongly Rayleigh property.

	Our third significant result is Theorem \ref{thm:notorder} which proves that the class $\mathcal{B}_d^{\Sigma}(\pp)$ is an antichain for the weak-association order, and it satisfies a weak minimality property with respect to $\preceq_{w-assoc}$. 
	Indeed, we prove that if $f_{\II} \in \mathcal{B}_d(\pp) \setminus \mathcal{B}_d^{\Sigma}(\pp)$, then either there exists $f^{\Sigma} \in \mathcal{B}_d^{\Sigma}(\pp)$ such that $f^{\Sigma}\preceq_{sm} f_{\II}$ or $f_{\II}$ and $f^{\Sigma}$ are not comparable.
	Last, we construct a chain of multivariate Bernoulli distributions for $\preceq_{sm}$ starting from extremal negative dependence to independence.
	A counterexample shows that this chain for $\preceq_{sm}$ is not a chain for $\preceq_{w-assoc}$.
	
	The paper is organized as follows. 
	Section~\ref{sec:neg-dep} and \ref{sec:Etremal} are devoted to providing the essential tools of negative dependence and extremal negative dependence in a common language and in the framework of multidimensional Bernoulli distributions.
	In Section~\ref{sec:main results}, we characterize the class of $\Sigma$-countermonotonicity and study the strongly Rayleigh property within this class.
	We conclude in Section~\ref{sect:WeakAssociationOrder} by proving the minimality properties of the $\Sigma$-countermonotonic class of distributions with respect to the $\preceq_{w-assoc}$ and $\preceq_{sm}$ orders.
	
	\section{Negative dependence}\label{sec:neg-dep}
	
	Negative dependence properties and dependence orders are the main ingredients to study negative dependence.
	Dependence orders are used to compare the strength of negative dependence between vectors in the same Fréchet class of joint distributions with identical one-dimensional marginal distributions according to a specific property.
	The strongest negative dependence property is the strongly Rayleigh property introduced in \cite{borcea2009negative}, and we present it separately in Section~\ref{sec:SR}.
	The strongly Rayleigh property is defined in the class of multidimensional Bernoulli distributions, which is  our object of study. We, therefore, introduce all the properties and orderings in the setting of multidimensional Bernoulli distributions. 
	
	Depending on the context, that is, whether working with random vectors $\II$, their cumulative distribution functions (cdfs) $F_{\II}$, or their probability mass functions (pmfs) $f_{\II}$, we write $\II \in \BBB_d$, $F_{\II} \in \BBB_d$, or $f_{\II} \in \BBB_d$, meaning that $\II$ has pmf $f_{\II}\in \BBB_d$. 
	Since to compare the strength of dependence of two random vectors they must have the same one dimensional distributions; we consider the Fréchet class  $\BBB_d(\pp)$.
	The authors of \cite{fontana2018representation} proved that the Fréchet class $\BBB_d(\pp)$ is a convex polytope, that is a convex hull of a finite set of points called extremal points; see \cite{de1997computational} for a standard reference on polytopes. Therefore, a pmf $f_{\II}$ belongs to $\BBB_d(\pp)$ if and only if there exists $\lambda_1,\dots,\lambda_{n_{\pp}} \geq 0$ summing up to one such that
	\begin{equation}
		f_{\II}(\ii) = \sum_{k=1}^{n_{\pp}} \lambda_k f_{\RRR_k}(\ii),
		\label{eq:PmfExtremalPoints}
	\end{equation}
	for every $\ii \in \{0,1\}^d$, where $f_{\RRR_k} \in \BBB_d(\pp)$, $k = 1,\dots,n_{\pp}$, are the extremal points of $\BBB_d(\pp)$, also called extremal pmfs. 
	Note that $f_{\RRR_k}$ is the pmf of a Bernoulli random vector $\RRR_k$, for $k = 1,\dots,n_{\pp}$, that we also refer to as an extremal point.

	\subsection{Negative dependence notions and orders}\label{Sec:Neg-dep}
	This section introduces the building blocks of the theory of negative dependence developed in the probability literature: dependence properties and orders.
	
	Both dependence properties and orders have to satisfy some conditions, see (P0)--(P9) in Definition 3.1 in \cite{kimeldorf1989framework} and (P1)--(P9) in Section 2.2.3 in \cite{joe1997multivariate}. 
	We only recall the conditions that are necessary in what follows, and we state them in our framework, that is, for negative dependence and Bernoulli vectors.
	A dependence order $\preceq_*$ has to satisfy the following conditions:
	
	\begin{enumerate}[label=(\subscript{O}{{\arabic*}})]
		\item \label{cond:partial} A dependence order is a partial order $\preceq_*$  on $\mathcal{B}_d(\pp)$.
		\item \label{cond:FC}  If $\II\preceq_*\II'$, then $\II$ and $\II'$ belong to the same Fréchet class $\mathcal{B}_d(\pp)$.
		\item \label{cond:bound} For any $\II\in \mathcal{B}_d(\pp)$, $F_{L}(\ii)\preceq_*F_{\II}(\ii)\preceq_* F_U(\ii)$ holds for every $\ii \in \{0,1\}^d$, where $F_{\II}$ is the joint cdf of $\II$, $F_L$ is the lower Fréchet bound of $\mathcal{B}_d(\pp)$, defined by
		\begin{equation*}
			F_L(i_1,\ldots, i_d) = \max(F_1(i_1) + \dots + F_d(i_d) - d+1, \, 0),
			\quad 
			(i_1,\dots, i_d) \in \{0,1\}^d,
		\end{equation*}
		and
		$F_U(\ii)$ is the upper Fréchet bound of $\mathcal{B}_d(\pp)$, defined by
		\begin{equation*}
			F_U(i_1,\ldots, i_d) = \min(F_1(i_1), \dots, F_d(i_d)),
			\quad 
			(i_1, \dots, i_d) \in \{0,1\}^d.
		\end{equation*}
	\end{enumerate}
	Note that \ref{cond:partial} is a consequence of (P0) in \cite{kimeldorf1989framework}.
	It is important to stress that while $F_U \in \mathcal{B}_d(\pp)$ is a distribution function in any dimension, $F_L$ is not always a distribution function in dimensions higher than two. This means that we always have a vector with maximal strength of dependence. Still, in dimensions higher than two, we do not have a vector corresponding to extremal negative dependence when the lower bound of the Fréchet class is not a distribution.
	
	A negative dependence property is a subset $\mathcal{N} \subseteq \mathcal{B}_d(\pp)$ of multidimensional distributions in $\mathcal{B}_d(\pp)$ such that the following conditions are satisfied:  
	\begin{enumerate}[label=(\subscript{P}{{\arabic*}})]
		\item \label{P_NC} If $f_{\II} \in \mathcal{N}$, then $F_{j_1,j_2}(i_{1}, i_{2})\leq F^{\perp}_{j_1,j_2}(i_{1}, i_{2})$, for all $(i_{1}, i_{2})\in \{0,1\}^2$ and $1\leq j_1 < j_2 \leq d$ where $F_{j_1,j_2}$ and $F^{\perp}_{j_1,j_2}$ are the $({j_1}, {j_2})$-bivariate marginal cdfs of $\II$ and $\II^{\perp} \in \mathcal{B}_d(\pp)$, respectively, and $\II^{\perp}$ has independent components.
		\item \label{cond:P2} Any random vector $\II^{\perp}$ with independent components has pmf belonging to $\mathcal{N}$.
		\item \label{cond:P3} The lower Fréchet bound $F_{\boldsymbol{L}}(\ii):=\max(F_1(i_1)+F_2(i_2)-1, 0) $, where $\ii=(i_1,i_2)\in\{0,1\}^2$ of $\mathcal{B}_2(\pp)$, belongs to $\mathcal{N}$ for any $\pp\in [0,1]^2$.
	\end{enumerate}
	
	Conditions \ref{P_NC}, \ref{cond:P2}, and \ref{cond:P3} are fulfilled by classes of functions defined as 
	\begin{equation}\label{eq:NDP-Ord}
		\mathcal{N}^*=\{f_{\II}\in \mathcal{B}_d(\pp): \II \preceq_* \II^{\perp}\},
	\end{equation}
	where $\preceq_*$ is a dependence order satisfying conditions \ref{cond:partial}, \ref{cond:FC}, and \ref{cond:bound}.
	A weaker but intuitive dependence notion is pairwise negative correlation (p-NC), that is, all pairs have non-positive correlation; this is implied by Condition \ref{P_NC} of a negative dependence property.
	The most famous notion of negative dependence is negative association (NA). For its properties and applications, see \cite{joag1983negative}.
	
	\begin{definition} \label{def:NA}
		A $d$-dimensional random vector $\II\in \mathcal{B}_d(\pp)$ is said to be negatively associated (NA) if for any two disjoint set $\Lambda_1, \Lambda_2 \subseteq \{1,\dots,d\}$ and two monotone increasing functions $h_1, h_2$, the following inequality holds
		\begin{equation*}
			E[h_1(I_j: j \in \Lambda_1)h_2(I_j: j \in \Lambda_2)] \leq E[h_1(I_j: j \in \Lambda_1)] E[h_2(I_j: j \in \Lambda_2)],
		\end{equation*}
		provided that the expectations are finite.
	\end{definition}
	
	NA property can be defined starting from the weak association order as follows, see Section 9.E of \cite{shaked2007stochastic} on pages 419--420.
	\begin{definition} \label{def:WeakAssOrder}
		Let $\boldsymbol{I}, \boldsymbol{K}\in \mathcal{B}_d(\pp)$. We say that $\boldsymbol{I}$ is smaller than $\boldsymbol{K}$ in the weak association order, denoted $\boldsymbol{I} \preceq_{w-assoc} \boldsymbol{K}$, if 
		\begin{equation}
			Cov(h_1(I_j: j \in \Lambda_1),h_2(I_j: j \in \Lambda_2))
			\leq
			Cov(h_1(K_j: j \in \Lambda_1),h_2(K_j: j \in \Lambda_2)),
			\label{eq:WeakAssociatedOrder}
		\end{equation}
		for all disjoint subsets $\Lambda_1,\Lambda_2 \subseteq \{1,\dots,d\}$ and for all monotone increasing functions $h_1$ and $h_2$, for which the covariances in \eqref{eq:WeakAssociatedOrder} are defined. 
	\end{definition}
	
	If $\II\in\mathcal{B}_d(\pp)$ is NA, and if $\II^{\perp}\in \mathcal{B}_d(\pp)$ is a vector with independent components, then $\II \preceq_{w-assoc}\II^{\perp}$ (see Remark 9.E.9 in \cite{shaked2007stochastic}), thus we can write
	$$\text{NA}=\{f_{\II}\in \mathcal{B}_d(\pp): \II \preceq_{w-assoc} \II^{\perp}\}.$$
	In the class of Bernoulli vectors, p-NC  is equivalent to requiring that all pairs are negatively associated, corresponding to pairwise negative association. p-NC is, therefore, the minimal requirement for negative dependence.
	A dependence notion weaker than NA is negative supermodular dependence defined through the supermodular order, see \cite{Bauerle2006stochastic}. A supermodular function is a function $\phi \colon \{0,1\}^d \to \RR$ such that $\phi(\ii)+\phi(\kk) \leq \phi(\ii \wedge \kk)+\phi(\ii \vee \kk)$, where $\wedge$ and $\vee$ respectively denotes the componentwise minimum and maximum operators.
	
	\begin{definition}
		Given two random vectors $\II=(I_1,\dots,I_d)$ and $\boldsymbol{K}=(K_1,\dots,K_d)$ with pmfs in the same Fréchet class, $\II$ is said to be smaller than $\KK$ in the supermodular order, denoted by \(\boldsymbol{I} \preceq_{sm} \boldsymbol{K}\), if
		\begin{equation*}
			E[\phi(\boldsymbol{I})]\le E[\phi(\boldsymbol{K})],
		\end{equation*}
		for all supermodular functions $\phi \colon \{0,1\}^d \to \RR$ such that the expectations are finite.
	\end{definition}
	Theorem 9.E.8 of \cite{shaked2007stochastic} proves that $\II \preceq_{w-assoc} \KK$ implies $\II \preceq_{sm} \KK$.
	The property of negative supermodular dependence is defined by comparing $\II$ in supermodular order with $\II^{\perp}$, see \cite{hu2000negatively}.

	\begin{definition} 
		A vector $\II \in \mathcal{B}_d(\pp)$ has the negative supermodular dependence property (NSD) if $\II \preceq_{sm}\II^{\perp}$, where $\II^{\perp}\in \mathcal{B}_d(\pp)$ has independent components, i.e.
		\begin{equation*}
			\text{NSD} = \{ f_{\II}\in \mathcal{B}_d(\pp): \II \preceq_{sm} \II^{\perp}\}.
		\end{equation*}
	\end{definition}
	In \cite{christofides2004connection}, the authors prove that NA implies NSD, that is also a consequence of the fact that $\II \preceq_{w-assoc} \KK$ implies $\II \preceq_{sm} \KK$.
	We summarize the main properties of NA and NSD, and the related orders they satisfy.
	The proof for $\preceq_{w-assoc}$ is in Theorem  9.E.7 and the proof for $\preceq_{sm}$ in Theorem 9.A.9 of \cite{shaked2007stochastic}.
	\begin{proposition} 
		Let $\boldsymbol{I}, \boldsymbol{K}\in \mathcal{B}_d(\pp)$ and $\preceq_{*}$ be either the weak association or supermodular order.
		\begin{enumerate}
			\item Closure under increasing transformation. If $\boldsymbol{I} \preceq_{*} \boldsymbol{K}$, then
			\begin{equation*}
				(g_1(I_1), \dots, g_d(I_d)) \preceq_{*} (g_1(K_1), \dots, g_d(K_d)),
			\end{equation*}
			whenever $g_j: \mathbb{R} \rightarrow \mathbb{R}$, $j = 1, \dots, d$, are increasing. 
			
			\item Closure under marginalization. If $\boldsymbol{I} \preceq_{*} \boldsymbol{K}$, then $ (I_j, j \in A) \preceq_{*} (K_j, j \in A)$,
			for any subset $A \subseteq \{1,\dots,d\}$. 
			
		\end{enumerate}
	\end{proposition}
	
	The following proposition for NSD property is provided in \cite{hu2000negatively} and for NA in \cite{joag1983negative}. 
	\begin{proposition} 
		Let $\boldsymbol{I} \in\mathcal{B}_d(\pp)$. The following statements hold:
		\begin{enumerate}
			\item If $\II$ is NA (NSD), then  $(g_1(I_1), \dots, g_d(I_d))$ is NA (NSD).
			\item If $\II$ is NA (NSD), then  $(I_j, j \in A)$ is NA (NSD).        
		\end{enumerate}
	\end{proposition}
	
	To complete the picture, we conclude this section with the definition of the negative lattice condition introduced in \cite{fortuin1971correlation}, also called the multivariate reverse rule of order 2 in \cite{karlin1980classes}. 
	\begin{definition}
		A random vector $\II \in \BBB_d(\pp)$ satisfies the negative lattice condition (NLC) if $\II$ has a pmf $f_{\II}$ satisfying
		\begin{equation}\label{eq:mtp2}
			f_{\II}(\ii) f_{\II}(\kk) \geq f_{\II}(\ii \wedge \kk)f_{\II}(\ii \vee \kk),
		\end{equation}
		for every $\ii, \kk \in \{0,1\}^d$.
	\end{definition}
	
	By reversing the inequality in \eqref{eq:mtp2}, we define the corresponding positive dependence property, called the positive lattice condition. 
	Positive lattice condition is a local condition easier to check than positive association, which implies positive association. 
	Unfortunately, the NLC does not imply negative association, and it is not implied by it. 
	The NLC does not imply p-NC and, therefore, is not a proper negative dependence property. 
	We mention this property because of its link with the strongly Rayleigh property, which we introduce in Section~\ref{sec:SR}, and with extremal negative dependence, which we introduce in Section~\ref{sec:Etremal}.

	\subsection{Strongly Rayleigh property}\label{sec:SR}
	
	In \cite{borcea2009negative}, the authors introduce a strong notion of negative dependence for Bernoulli vectors through a property of their probability generating function (pgf). Let $\II\in\BBB_d(\pp)$, its pgf is defined as
	\begin{equation*}
		\mathcal{P}_{\II}(\zz) 
		=
		E[\zz^{\II}], 
	\end{equation*}
	where $\zz = (z_1,\dots,z_d) \in \mathbb{C}^d$ and $\zz^{\II} := \prod_{j=1}^d z_j^{I_j}$.
	The pgf $\PPP_{\II}$ of a Bernoulli random vector $\II$ is a multi-affine polynomial with positive real coefficients. 
	A multi-affine polynomial $\mathcal{P}$ is a polynomial in which each variable has a degree at most one.
	We need a preliminary notion on polynomials to introduce the strongly Rayleigh property.
	
	\begin{definition}\label{def:stab}
		A non-zero polynomial $\mathcal{P} \in \mathbb{C}[z_1,\dots,z_d]$ is stable if $\mathcal{P}(z_1,\dots,z_d) \neq 0$ whenever $\zz = (z_1,\dots,z_d) \in \HHH^d$, where $\HHH = \{ z \in \mathbb{C} : \Im(z)>0 \}$, where $\Im(z)$ denotes the imaginary part of the complex number $z$. If the coefficients of $\mathcal{P}$ are real, then $\mathcal{P}$ is real stable.
	\end{definition}
	In Theorem 5.6 of \cite{branden2007polynomials}, it is shown that a multi-affine polynomial $\PPP(\zz)$ with real coefficients is real stable if and only if, for every $1 \leq j_1, j_2 \leq d$,
	\begin{equation} \label{eq:CondStrongRayleigh}
		\frac{\partial \PPP}{\partial z_{j_1}} (\xx) \frac{\partial \PPP}{\partial z_{j_2}} (\xx) \geq \frac{\partial^2 \PPP}{\partial z_{j_1} \partial z_{j_2}} (\xx) \PPP(\xx), \quad \text{for all } \xx \in \RR^d.
	\end{equation}
	See \cite{wagner2011multivariate} for more details on stable and multi-affine polynomials.
	
	We can now define the strongly Rayleigh property.
	\begin{definition}
		A vector $\II\in\mathcal{B}_d(\pp)$ --- or its pmf $f_{\II} \in \BBB_d(\pp)$ --- satisfies the strongly Rayleigh property if its pgf function $\mathcal{P}_{\II}$ is real stable.
	\end{definition}
	
	It is easy to check that p-NC is equivalent to requiring that \eqref{eq:CondStrongRayleigh} holds for $\xx=\boldsymbol{1}$, where $\boldsymbol{1}$ is a vector of all ones; this immediately clarifies that the strongly Rayleigh property implies p-NC. 
	In Section 4.2 of \cite{borcea2009negative}, it is proven that the strongly Rayleigh property implies NA and that NA does not imply it. Therefore, it is the strongest negative dependence property. In the same paper, it is shown that the strongly Rayleigh property is stronger than the NLC.
	
	The strongly Rayleigh property is a negative dependence property that is not defined as in \eqref{eq:NDP-Ord} through a dependence order.
	Nevertheless, in \cite{borcea2009negative} (Definition 4.1, pages 545--546), the authors introduce a partial order $\unlhd$ on the set of strongly Rayleigh pmfs. 
	Their Proposition 4.12 proves that $\unlhd$ implies the usual stochastic order $\preceq_{st}$. 
	See $\cite{shaked2007stochastic}$ for the definition of the usual stochastic order and as a standard reference on stochastic orders. 
	It is well known that, see Theorem 6.B.19 of \cite{shaked2007stochastic}, any two random vectors $\II$ and $\II'$ with the same means and such that $\II\preceq_{st}\II'$  have the same joint distribution. 
	This means that two vectors $\II$ and $\II'$ in the same Fréchet class $\mathcal{B}_d(\pp)$ have the same distribution or are not comparable in $\unlhd$. 
	This implies that $\unlhd$ is not a dependence order as defined in \cite{joe1997multivariate} since it does not satisfy Conditions \ref{cond:FC} and \ref{cond:bound}. 
	Consequently, within a Fréchet class, it is impossible to compare the strength of the strongly Rayleigh property. 
	This means that $\unlhd$ is not the right partial order to define a class of extremal (meaning minimal with respect to the corresponding partial order) negative dependence vectors. 
	The best we can do is to define a class of extremal negative dependence distributions with respect to the traditional dependence (partial) orders, such as $\preceq_{w-assoc}$ or $\preceq_{sm}$, and look for the distribution within the extremal dependence class that satisfies the strongly Rayleigh property.

	\begin{remark}\label{ex:lin}
		Some examples of real stable polynomials are the monomials $\prod_{i=1}^d z_i$, for $d \in \{1,2,\dots\}$, and linear polynomials $\sum_{i=1}^d a_i z_i$ with coefficients $a_i \geq 0$. The polynomial $1 - z_1z_2$ is stable, in fact we observe that if $(z_1,z_2) \in \mathcal{H}^2$, then $z_1z_2<0$ or $z_1z_2\in \mathbb{C}$. For the same reason, we notice that  $1+z_1z_2$ is not stable.    
		For every $d \in \{1,2,\dots\}$ and $m \in \{0,\dots,d\}$, the elementary symmetric polynomial $\EEE_{d,m}$ is real stable, where
		\begin{equation*}
			\mathcal{E}_{d,m}(z_1,\dots,z_d) = 
			\sum_{1 \leq j_1 < \dots < j_m \leq d} z_{j_1} \dots z_{j_m}, \quad \text{for } m \in \{1,\dots,d\},
		\end{equation*}
		and $\EEE_{d,0} \equiv 1$.
	\end{remark}
	
	Lemma \ref{lemma:StabilityPreservers} directly follows from points (a), (b), (c), and (d) of Lemma 2.4 in \cite{wagner2011multivariate}.
	\begin{lemma} \label{lemma:StabilityPreservers}
		Let $\II \in \BBB_d(\pp)$ be a $d$-dimensional Bernoulli random vector with pmf $f_{\II}$ and pgf $\PPP_{\II}$.
		If $\II$ satisfies the strongly Rayleigh property, then:
		\begin{enumerate}[label=(\alph*)]
			\item \label{point:permutation} $(I_{\sigma(1)}, \dots, I_{\sigma(d)})$ satisfies the strongly Rayleigh property, for every permutation $\sigma$.
			\item \label{point:scaling} $c\PPP(a_1z_1,\dots,a_dz_d)$ is stable, for every $c \in \mathbb{C}$, $c \neq 0$, and $a_1,\dots,a_d > 0$.
			\item \label{point:diagonalization} The $(d-1)$-dimensional random vector $(I_1+I_2,I_3,\dots,I_d)$ satisfies the strongly Rayleigh property.
			\item \label{point:specialization} For any $k \in \{1,\dots,d\}$ and for every $1 \leq j_1 < \dots < j_k \leq d$, the Bernoulli random vector $(I_{j_1}, \dots, I_{j_k}) \in \BBB_k$ satisfies the strongly Rayleigh property, i.e., the strongly Rayleigh property is closed under marginalization.
		\end{enumerate}
	\end{lemma}
	Notice that the proof of Point \ref{point:specialization} follows from the specialization property of Lemma 2.4 in \cite{wagner2011multivariate} with $a=1$.
	Obviously, from the permutation property in Point \ref{point:permutation}, Point \ref{point:diagonalization} holds for the sum of $I_{j_1}$ and $I_{j_2}$ for any index $j_1$ and $j_2$, and not only for $j_1=1$ and $j_2=2$.
	
	The following Theorem \ref{thm:LinearComb_StrongRayleigh} relies on Theorem 1.6 in \cite{borcea2010multivariate}  and Theorem 4.11 in \cite{borcea2009negative} and provides conditions for linear combinations of polynomials to be real stable.
	\begin{theorem} \label{thm:LinearComb_StrongRayleigh}
		Let $\mathcal{P}_1, \mathcal{P}_2 \in \RR[z_1,\ldots, z_d]$ such that $\PPP_1 \not\equiv 0$ and $\PPP_2 \not\equiv 0$. The following are equivalent:
		\begin{enumerate}
			\item All non-zero polynomials in the set $\{\gamma_1 \PPP_1 + \gamma_2 \PPP_2: \gamma_1,\gamma_2 \in \RR\}$ are real stable.
			\item Either $\mathcal{P}_1+z_{d+1}\mathcal{P}_2 \in \RR[z_1,\ldots, z_d,z_{d+1}]$ or $\mathcal{P}_2+z_{d+1}\mathcal{P}_1 \in \RR[z_1,\ldots, z_d,z_{d+1}]$ is real stable, where $z_{d+1}$ is a new indeterminate.
		\end{enumerate}
	\end{theorem}
	It is easy to check, for example, that $\EEE_{d,m+1} + z_{d+1}\EEE_{d,m} = \EEE_{d+1,m+1}$.
	Therefore, two consecutive elementary symmetric polynomials verify the conditions of Theorem \ref{thm:LinearComb_StrongRayleigh}.
	
	As discussed in Section 4.3 of \cite{borcea2009negative}, if the pgfs of two Bernoulli random vectors $\II_1, \II_2 \in \BBB_d$ satisfy the conditions in Theorem \ref{thm:LinearComb_StrongRayleigh}, then $\II_1 \unlhd \II_2$ or $\II_2 \unlhd \II_1$. 
	We have already mentioned that this implies that $\II_1$ and $\II_2$ have the same distribution or that their pmfs belong to different Fréchet classes. 
	Hence, two different Bernoulli pmfs of the same Fréchet class cannot be such that their pgfs satisfy the conditions of Theorem \ref{thm:LinearComb_StrongRayleigh}.
	
	In Figure \ref{fig:DefScheme}, we summarize the many links between dependence properties and orders and the strongly Rayleigh property.

	\section{Extremal Negative dependence}\label{sec:Etremal}

	The extremal negative dependence in dimension two corresponds to the lower Fréchet bound, that is, the distribution of a bivariate vector with countermonotonic components.
	
	\begin{definition} \label{def:countermonotonicity}
		The random vector $\II=(I_1, I_2)$ defined on a probability space $(\Omega, \mathcal{F}, \mathbb{P})$ is countermonotonic if
		\begin{equation*}
			(I_1(\omega)-I_1(\omega') )
			(I_2(\omega)-I_2(\omega')) \leq 0,
			\quad \text{for } (\mathbb{P}\times\mathbb{P}) \text{-almost every } (\omega,\omega') \in \Omega^2.
		\end{equation*}
	\end{definition}
	See, for example, \cite{dhaene1999safest} and \cite{puccetti2015extremal} for details about the concept of countermonotonicity. The simple extension of countermonotonicity is pairwise countermonotonicity, which requires each pair of variables within a vector to be countermonotonic. Pairwise countermonotonicity was introduced in \cite{dall1972frechet} and recently studied in \cite{lauzier2023pairwise}. We do not examine this property because we prove in Section~\ref{sec:main results} that in the class $\mathcal{B}_d(\pp)$, a pairwise countermonotonic vector exists if and only if the lower Fréchet bound is a distribution and hence that it is the distribution of the pairwise countermonotonic vector.
	
	We consider here three notions of extremal negative dependence introduced in the literature to generalize countermonotonicity in dimensions higher than two: minimality in convex sums, joint mixability, and $\Sigma$-countermonotonicity; see \cite{puccetti2015extremal} and references therein for a discussion on these notions in a general framework. 
	In our setting, minimality in convex sums consists, more specifically, of searching for vectors $\II=(I_1,\ldots, I_d)$ whose sums $S_{\II}=\sum_{j=1}^dI_j$ are minimal in convex order, that we recall in the following definition.
	
	\begin{definition}
		Given two random variables $X$ and $Y$, we say that $X$ is smaller than $Y$ under the convex order, denoted by $X \preceq_{cx} Y$, if $E[\phi(X)] \leq E[\phi(Y)]$ for every convex function $\phi:\RR \to \RR$, for which the expectations are finite.
	\end{definition}
	
	We recall Definition~3.3 of extremal negative dependence from \cite{puccetti2015extremal} . 
	
	\begin{definition} [$\Sigma_{cx}$-smallest element]
		\label{def:SigmaConvex}
		A random vector $\boldsymbol{I} = (I_1,\dots,I_d)$ with pmf $f_{\II}$ in the class $\mathcal{B}_d(\pp)$ is a {$\Sigma_{cx}$-smallest element} of $\mathcal{B}_d(\pp)$ if 
		\begin{equation*}
			\sum_{j=1}^d I_j \preceq_{cx} \sum_{j=1}^d I_j', 
		\end{equation*}
		for all $\boldsymbol{I}' = (I_1',\dots,I_d')$ with pmf in the class $\mathcal{B}_d(\pp)$.
	\end{definition}
	
	\begin{remark}
		Consider a vector $\II$ that is minimal in $\preceq_{sm}$ in a given Fréchet class $\mathcal{B}_d(\pp)$, whenever it exists. 
		Then, it is also a $\Sigma_{cx}$-smallest element of $\mathcal{B}_d(\pp)$ since any function $\phi(\xx) = g(\sum_{i=1}^dx_i)$ is supermodular, if $g$ is convex. 
	\end{remark}
	A $\Sigma_{cx}$-smallest element in general does not always exist. However, as hinted in Lemma 3.1 of \cite{bernard2017robust}, it does in the class of Bernoulli random variables. Since any discrete distribution with support on $\{0,\ldots, d\}$ can be constructed as the distribution of a sum of Bernoulli variables in many ways (\cite{fontana2024bernoulli}), usually several $\Sigma_{cx}$-smallest elements exist.
	
	Given any vector $\xx=(x_1,\dots,x_d) \in \RR^d$, we denote by $x_{\bullet} = \sum_{j=1}^d x_j$ the sum of its components.
	Let $\II\in \mathcal{B}_d$, then the sum $S_{\II}=\sum_{i=1}^dI_i$ has a discrete distribution with support on $\{0,\ldots, d\}$.  
	The vector $\II$ is a $\Sigma_{cx}$-smallest element in $\mathcal{B}_d(\pp)$ if $S_{\II}$ is minimal in convex order in $\mathcal{D}_d(p_{\bullet})$, where $\mathcal{D}_d(p_{\bullet})$ is the class of discrete distributions with support on $\{0,\ldots, d\}$ and mean $p_{\bullet}=\sum_{j=1}^dp_j.$
	
	For any $p_{\bullet} \in (0,d)$, the minimal distribution in convex order $s_{p_{\bullet}}$ in $\mathcal{D}_d(p_{\bullet})$  is known (see  e.g. \cite{bernard2017robust} and \cite{fontana2024high}), and it is given by
	\begin{equation} \label{eq:ExtrPoints_D}
		s_{p_{\bullet}}(y)
		=
		\begin{cases}
			{m+1-p_{\bullet}}, &y = m
			\\
			p_{\bullet}-m, &y = m+1
			\\
			0 &\text{otherwise}
		\end{cases},
	\end{equation}
	where $m$ is the integer smaller than or equal to $p_{\bullet}$.
	The class of $\Sigma_{cx}$-smallest elements is the class of Bernoulli vectors $\II$ with pmf $f_{\II}$ such that the pmf of $S_{\II} = \sum_{j=1}^dI_j$ is equal to $s_{p_{\bullet}}$, for some $p_{\bullet} \in (0,1)$. 
	
	From Theorem 2.1 of \cite{fontana2024bernoulli}, it follows that the class of $\Sigma_{cx}$-smallest elements is a convex polytope, and hence any $\Sigma_{cx}$-smallest element has pmf that is a convex combination of extremal points, that are known analytically. 
	Notice that if $p_{\bullet}\leq1$ or $p_{\bullet}\geq d-1$, then $m=0$ or $m=d-1$, and $s_{p_{\bullet}}$ has support on $\{0,1\}$ or $\{d-1,d\}$ and the polytope of $\Sigma_{cx}$-smallest elements has only one point, the lower Fréchet bound, that is proven to be a distribution in these cases in \cite{dhaene1999safest}.
	
	For $m \in \{0,\dots,d \}$, let the subset $\mathcal{A}_{d,m}$ of $\{0,1\}^d$ be defined by
	\begin{equation*}
		\mathcal{A}_{d,m} 
		=  \{\ii = (i_1,\dots,i_d) \in \{0,1\}^d : i_1 + \dots + i_d = m \}.
	\end{equation*}
	From \eqref{eq:ExtrPoints_D}, $\II \in \mathcal{B}_d(\pp)$ is a $\Sigma_{cx}$-minimal element if and only if it has support on $\mathcal{A}_{d,m} \cup \mathcal{A}_{d,m+1}$ when $p_{\bullet} \in (m, m+1)$. 
	If $p_{\bullet}$ is integer, any minimal $\Sigma_{cx}$-smallest vector has support on $\mathcal{A}_{d,p_{\bullet}}$ and it is a joint mix, according to the following definition.
	
	\begin{definition}\label{def:j-mix}
		A Bernoulli random vector $\II \in \BBB_d(\pp)$ is a joint mix if there exists $m \in \mathbb{R}$ such that $\mathbb{P}\left(S_{\II} = m \right) = 1$.
	\end{definition}
	From Definition \ref{def:j-mix}, it is clear that if $p_{\bullet}$ is not integer the Fréchet class does not admit any joint mix element.
	The last notion of extremal negative dependence is a generalization of countermonotonicity weaker than minimality in convex sums, but which exists in each Fréchet class --- not only for Bernoulli random variables.
	
	\begin{definition}
		A Bernoulli random vector $\II \in \BBB_d(\pp)$ is {$\Sigma$-countermonotonic} if for every subset $\Lambda \subseteq \{1,\dots,d\}$, the random variables $\sum_{j \in \Lambda} I_j$ and $\sum_{j \notin \Lambda} I_j$ are countermonotonic.
	\end{definition}
	We use the convention $\sum_{j \in \varnothing} I_j = 0$.
	The authors of \cite{puccetti2015extremal} proved in Theorem 3.8 that $\Sigma_{cx}$-smallest vectors are $\Sigma$-countermonotonic. 
	
	In Figure \ref{fig:DefScheme}, together with the links among dependence properties and orders, we also provide a summary of links among negative dependence notions and extremal dependence.
	For completeness, we conclude this section with the following remark about the connection between strongly Rayleigh vectors $\II$ and their sum $S_{\II}$.
	
	\begin{remark}
		A univariate distribution is strongly Rayleigh if and only if all the zeros of its pgf are real, see \cite{borcea2010multivariate}.  
		We can easily prove that if a univariate discrete distribution defined on the set $\{0,1,\dots,d\}$ is minimal under the convex order, it satisfies the strongly Rayleigh property.
		Indeed, for $p_{\bullet} \in [m,m+1)$, the pmf of $\mathcal{D}_d(p_{\bullet})$ that is minimal under the convex order is $s_{p_{\bullet}}$ and its pgf is $\mathcal{P}(z) = s_{p_{\bullet}}(m) z^m + s_{p_{\bullet}}(m+1) z^{m+1}$. 
		$\mathcal{P}(z)$ has only real zeros, and thus it is real stable.
		Notice that, from Point \ref{point:diagonalization} of Lemma \ref{lemma:StabilityPreservers}, the fact that $s_{p_{\bullet}}$ is a univariate strongly Rayleigh pmf is a necessary condition for a multidimensional $\Sigma_{cx}$-smallest pmf to be strongly Rayleigh.
	\end{remark}
	
	
	\begin{figure} [tb]
		\centering 
		
		\begin{tikzpicture}[scale=0.8, transform shape,
			roundnode/.style={circle, draw=green!60, fill=green!5, very thick, minimum size=7mm},
			squarednode/.style={rectangle, draw=black, very thick, minimum size=5mm},
			]
			\node[squarednode] (O) at (0,9) {Orders};
			\node[squarednode] (ExtrMin) at (5,9) {Extremal Minimality};
			\node[squarednode] (N) at (10,9) {Properties};
			
			\node (W-NA) at (0,6) {$\II \preceq_{w-assoc} \II'$};
			\node (SM) at (0,3) {$\II \preceq_{sm} \II'$};
			\node (Cov) at (0,0) {$Cov(I_{j_1},I_{j_2}) \leq Cov(I'_{j_1},I'_{j_2}), \forall j_1,j_2$};
			\node (SigmaCX) at (2,1.5) {$\sum I_j \preceq_{cx} \sum I'_j$};
			\draw[->] (W-NA) -- (SM);
			\draw[->] (SM) -- (Cov);
			\draw[->] (SM) -- (SigmaCX);
			\draw[->] (SigmaCX) -- (Cov) node[midway, sloped] {$\not\not$};

			\node (CTM) at (5,5) {min $\Sigma_{cx}$ $=$ $\Sigma$-ctm};

			\node (SR) at (10,8) {Strongly Rayleigh};
			\node (NA) at (10,6) {NA};
			\node (NSM) at (10,3) {NSD};
			\node (p-NC) at (10,0) {p-NC};
			\draw[->] (SR) -- (NA);
			\draw[->] (NA) -- (NSM);
			\draw[->] (NSM) -- (p-NC);
			\draw[->] (SR) -- (CTM) node[midway, sloped]{$\not\not$};
			\draw[->] (CTM) -- (p-NC) node[midway, sloped]{$\not\not$};
			\draw[-,dash pattern=on 3pt off 2pt] (SigmaCX) -- (CTM);

			\node (N_LC) at (7,0) {NLC}; 
			\draw[->] (CTM) -- (N_LC); 
			\draw[->] (SR) -- (N_LC);
			
		\end{tikzpicture}

		\caption{\emph{Links among dependence notions and extremal dependence. The arrow $\to$ means strictly implies, and $\not\to$ means does not imply.  The dashed line means that there is a link but not an implication.
		}}
		\label{fig:DefScheme}    
	\end{figure}
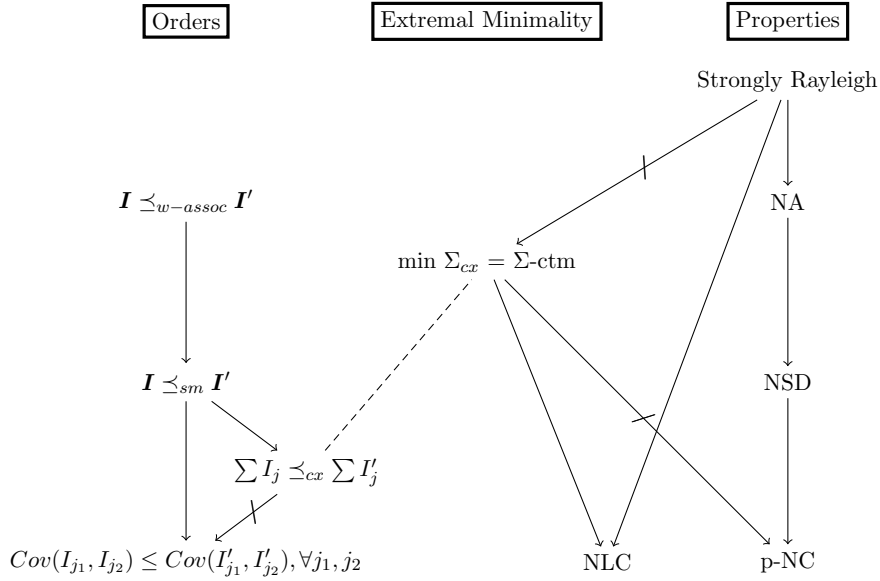


	\section{Main results}\label{sec:main results}
	This section characterizes the class of $\Sigma$-countermonotonic Bernoulli random vectors and their relationship with the strongly~Rayleigh property.
	
	\subsection{Characterization of $\Sigma$-countermonotonicity}\label{sec:characterization}
	
	In this section, Theorem \ref{thm:EquivalentDef} proves that $\Sigma$-countermonotonic Bernoulli random vectors are $\Sigma_{cx}$-smallest. This result is significant for two reasons. 
	Firstly, the two extensions of the lower Fréchet bound in dimensions higher than two lead to the same class of multidimensional Bernoulli distributions; secondly, this class of multidimensional Bernoulli distributions is a convex polytope.
	Then, we prove that $\Sigma$-countermonotonic Bernoulli pmfs are not always strongly~Rayleigh. 
	
	\begin{theorem} \label{thm:EquivalentDef}
		Let $\II \in \BBB_d(\pp)$.
		Then, the following statements are equivalent:
		\begin{enumerate}
			\item \label{Point1} $\II$ is $\Sigma_{cx}$-smallest.
			\item $\II$ is $\Sigma$-countermonotonic.
			\item $I_h$ and $\sum_{\underset{j \neq h}{j=1}}^d I_j$ are countermonotonic, for every $h \in \{1,\dots,d\}$.
		\end{enumerate}
	\end{theorem}
	
	\begin{proof}
		Theorem 3.8 of \cite{puccetti2015extremal} states that 1$\implies$2, and it is clear that 2$\implies$3.
		We will prove that if $\II = (I_1,\dots,I_d)$ is a Bernoulli random vector, then 3$\implies$1.
		Let $\II \in \BBB_d(\pp)$ be a Bernoulli random vector such that $I_h$ and $\sum_{j \neq h} I_j$ are countermonotonic, for every $h \in \{1,\dots,d\}$, and let $f_{\II}$ be its pmf.
		Let $p_j=E[I_j]$, $j=1,\dots,d$, and $p_{\bullet} = \sum_{j=1}^d p_j$.
		Suppose that the distribution of $\II$ is not a $\Sigma_{cx}$-smallest element, or, in other terms, that the distribution $s \in \DDD(p_{\bullet})$ of the sum $S_{\II} = \sum_{j=1}^d I_j$ is different from $s_{p_{\bullet}}$ given in \eqref{eq:ExtrPoints_D}.
		Therefore, there exist $m_1,m_2 \in \{0,1,\dots,d\}$, with $m_2 - m_1 \geq 2$, such that $s(m_1)>0$ and $s(m_2)>0$.   
		This means that there exist $\ii_1 \in \AAA_{d,m_1}$ and $\ii_2 \in \AAA_{d,m_2}$ such that $f_{\II}(\ii_1 ) > 0$ and $f_{\II}(\ii_2 ) > 0$.
		Since $m_2 \geq m_1 + 2$, there exist $h^* \in \{1,\dots,d\}$ such that $i_{1,h^*} = 0$ and $i_{2,h^*} = 1$.
		In the context of Definition \ref{def:countermonotonicity}, since $I_{h^*}$ and $\sum_{j \neq h^*} I_j$ are countermonotonic, we have
		\begin{equation*}
			\bigg( \sum_{j \neq h^*} I_j(\omega)-\sum_{j \neq h^*} I_j(\omega') \bigg)
			\bigg(I_{h^*}(\omega)-I_{h^*}(\omega') \bigg) \leq 0,
			\quad \text{for } (\mathbb{P}\times\mathbb{P}) \text{-almost every } (\omega,\omega') \in \Omega^2.
		\end{equation*}
		By considering $(\omega,\omega') \in \Omega^2$ such that $\II(\omega) = \ii_1$ and $\II(\omega') = \ii_2$, we have $(m_1 - (m_2-1) )(0-1) \leq 0$, that implies $m_2 - m_1 \leq 1$.
		However, the latter contradicts the initial assumption that $m_2 - m_1 \geq 2$. This is absurd and therefore $\II$ is $\Sigma_{cx}$-smallest.
	\end{proof}
	
	\begin{remark}
		The third statement of Theorem \ref{thm:EquivalentDef} is the stopping condition of the rearrangement algorithm in \cite{puccetti2012computation}. 
		Since the three definitions are equivalent, in the Bernoulli framework, the rearrangement algorithm always returns a $\Sigma_{cx}$-smallest pmf.
	\end{remark}
	
	\begin{remark}\label{remark:astar}
		Let us denote by $\BBB_d^{\Sigma} \subseteq \BBB_d$ the class of $\Sigma$-countermonotonic multivariate Bernoulli distributions of $\mathcal{B}_d$.
		Since  $f\in \mathcal{B}_d^{\Sigma}(\pp)$ is $\Sigma_{cx}$-minimal in a Fr\'echet class $\mathcal{B}_d(\pp)$ for a $\pp\in(0,1)^d$, it has support on  $\AAA^* =\mathcal{A}_m\cup\mathcal{A}_{m+1}$, where $m$ is the integer smaller than or equal to $p_{\bullet}$. 
	\end{remark}
	If $\II \in \BBB_d^{\Sigma}$, then there exists $m \in \{0,1,\dots,d-1\}$ such that the corresponding pgf is of the form
	\begin{equation}\label{pgfS}
		\mathcal{P}_{\II}(\zz) = \sum_{\ii \in \mathcal{A}_{d,m}} a_{\ii} \boldsymbol{z}^{\ii}
		+ 
		\sum_{\ii \in \mathcal{A}_{d,m+1}} a_{\ii} \boldsymbol{z}^{\ii},
	\end{equation}
	where $\zz^{\ii} = \prod_{j=1}^d z_j^{i_j}$. 
	If $\II$ is joint mix, then \eqref{pgfS} simplifies to
	\begin{equation}\label{pgfS2}
		\mathcal{P}_{\II}(\zz) = \sum_{\ii \in \mathcal{A}_{d,m}} a_{\ii} \boldsymbol{z}^{\ii},
	\end{equation}
	for $m\in \{0,\ldots, d\}$.
	
	We now characterize the geometrical structure of $\mathcal{B}_d^{\Sigma}$ and its intersection with the Fr\'echet class $\mathcal{B}_d(\pp)$, with $\pp\in (0,1)^d$.
	Proposition \ref{prop:SigmaCTM} proves that the class $\mathcal{B}_d^{\Sigma}$ of $\Sigma$-countermonotonic Bernoulli random vectors is the union of standard simplexes. 
	
	\begin{proposition} \label{prop:SigmaCTM}
		There is a one to one map between the class $\BBB_d^{\Sigma}$ and the union of the standard simplexes, as follows 
		\begin{equation*}
			\BBB_d^{\Sigma}\longleftrightarrow \cup_{m=0}^{d-1}\Delta_{n_m, 1},
		\end{equation*}
		where $\Delta_{n_m, 1}$ is the standard $n_m$-simplex where $n_m=\binom{d+1}{m+1}$.
	\end{proposition}
	
	\begin{proof}
		From Theorem \ref{thm:EquivalentDef}, if $\II$ is a Bernoulli random vector with pmf $f_{\II} \in \BBB_d^{\Sigma}$, then the pmf $s$ of the sum of its components $S_{\II}$ has support on two consecutive points. 
		Therefore, there exists $m$ such that $f_{\II}$ has support on $\mathcal{A}_{d,m} \cup \mathcal{A}_{d,m+1}$. 
		The space of pmfs on $\mathcal{A}_{d,m} \cup \mathcal{A}_{d,m+1}$ is in one to one correspondence with the standard simplex in dimension $\binom{d}{m}+\binom{d}{m+1}=\binom{d+1}{m+1}$.
	\end{proof}

	The strength of dependence of two vectors can be compared if they have the same one-dimensional marginal distributions, that is, if they belong to the same Fréchet class $\mathcal{B}_d(\pp)$. 
	From Theorem \ref{thm:EquivalentDef}, we can  prove that the class $\mathcal{B}_d^{\Sigma}(\pp):=\mathcal{B}_d^{\Sigma}\cap \mathcal{B}_d(\pp)$ is a convex polytope.
	\begin{proposition}\label{prop:polcx}
		The class $\mathcal{B}_d^{\Sigma}(\pp)$ is a convex polytope. Its extremal points are the extremal points of $\mathcal{B}_d(\pp)$ that are $\Sigma$-countermonotonic.
	\end{proposition}
	\begin{proof}
		If $f_{\II} \in \BBB_d(\pp)$, then there exist $\alpha_1,\dots,\alpha_{n_{d}(\pp)} \geq 0$ summing up to one such that $f_{\II} = \sum_{k=1}^{n_{d}(\pp)} \alpha_k f_{\RRR_k}$, where $f_{\RRR_k}$, $k = 1,\dots,{n_{d}(\pp)}$, are the extremal points of $\BBB_d(\pp)$ and ${n_{d}(\pp)}$ is their number.
		A pmf $f_{\II} \in \BBB_d(\pp)$ is a $\Sigma_{cx}$-smallest element of $\BBB_d(\pp)$ if and only if its support is contained in $\AAA_{d,p_{\bullet}}$, if $p_{\bullet}$ is integer, or if it is contained in $\AAA_{d,m} \cup \AAA_{d,m+1}$, if $p_{\bullet} \in (m,m+1)$.
		Therefore, $f_{\II} \in \BBB_d(\pp)$ is a $\Sigma_{cx}$-smallest element of $\BBB_d(\pp)$ if and only if $\alpha_k = 0$ for every $k \in \{1,\dots,{n_{d}(\pp)}\}$ such that $f_{\RRR_k}$ is not $\Sigma_{cx}$-smallest.
		Since every Fréchet class admits a $\Sigma_{cx}$-smallest element, this also implies that there exists a $\Sigma_{cx}$-smallest extremal point.
		In other terms, $f_{\II} \in \BBB_d(\pp)$ is $\Sigma_{cx}$-smallest if and only if there exist $\alpha_1,\dots,  \alpha_{{n^{\Sigma}_{d}(\pp)}} \geq 0$, $n^{\Sigma}_{d}(\pp) \in \{1,\dots,{n_{d}(\pp)}\}$, with $\sum_{k=1}^{n^{\Sigma}_{d}(\pp)} \alpha_k = 1$, such that $f_{\II} = \sum_{k=1}^{n^{\Sigma}_{d}(\pp)} \alpha_k f_{\RRR_k}^{\Sigma}$, where $f_{\RRR_k}^{\Sigma}$, $k = 1,\dots,n^{\Sigma}_{d}(\pp)$, are the $\Sigma_{cx}$-smallest extremal points of $\BBB_d(\pp)$.
		The result follows by Theorem \ref{thm:EquivalentDef} with the equivalence of $\Sigma_{cx}$-minimality and $\Sigma$-countermonotonicity.
	\end{proof}

	We now study the link between $\mathcal{B}_d^{\Sigma}(\pp)$ and the strongly Rayleigh property. 
	The strongly Rayleigh property does not imply the $\Sigma$-countermonotonicity property; in fact, any independent Bernoulli random vector has the strongly Rayleigh property. 
	The following counterexample shows that the reverse implication is not true either. 
	In particular, $\Sigma$-countermonotonicity does not implies p-NC. 
	\begin{example} \label{ex:ctexample1}
		Consider a Bernoulli random vector $\II =  (I_1,I_2,I_3)$ with $f_{\II} \in \BBB_3(\frac{2}{5})$ whose values are provided in the following table:
		\begin{table}[H]
			\centering
			\begin{tabular}{c|cccccccc}
				\toprule
				$\ii$ & (0,0,0) & (1,0,0) & (0,1,0) & (1,1,0) & (0,0,1) & (1,0,1) & (0,1,1) & (1,1,1) \\ 
				\midrule
				$f_{\II}(\ii)$ & 0 & $\tfrac{1}{5}$ & $\tfrac{1}{5}$ & $\tfrac{1}{5}$ & $\tfrac{2}{5}$ & 0 & 0 & 0 \\ 
				\bottomrule
			\end{tabular}
		\end{table}
		We observe that $\II$ is $\Sigma$-countermonotonic; but $\II$ does not satisfy the p-NC property, indeed $\text{Cov}(I_1,I_2) = \tfrac{1}{25}>0$.
		It follows that $\Sigma$-countermonotonicity does not imply all the negative dependence properties stronger than p-NC, including the strongly Rayleigh property.
	\end{example}
	Example \ref{ex:ctexample1} shows that $\Sigma$-countermonotonic Bernoulli random vectors can have positively correlated pairs of random variables.
	This can be explained observing that if $\II\in \mathcal{B}_d(\pp)$, the sum of all the cross moments of order two depends only on the distribution of the sum $S_{\II}$, i.e., 
	\begin{equation} \label{eq:CrossMomentsSum}
		\sum_{1\leq j_1 < j_2 \leq d} E(I_{j_1}I_{j_2})
		=
		\frac{1}{d(d-1)}\sum_{y=2}^d y(y-1) \mathbb{P}(S_{\II}=y) = E[\psi_2(S_{\II})],
	\end{equation}
	where $\psi_2(y) = \tfrac{y(y-1)}{d(d-1)}$.
	The expression in \eqref{eq:CrossMomentsSum} takes the same value for every Bernoulli vector $\II \in \BBB_d^{\Sigma}(\pp)$ since
	\begin{equation*}
		\sum_{1\leq j_1 < j_2 \leq d} E(I_{j_1}I_{j_2})
		=
		\frac{1}{d(d-1)}\sum_{y=2}^d y(y-1) s_{p_{\bullet}}(y) = m(2 p_{\bullet} - m -1),
	\end{equation*}
	where $m$ is the highest integer smaller than or equal to $p_{\bullet}$.
	Since $\psi_2(y)$ is a convex function, this is the minimum value in the Fréchet class $\BBB_d(\pp)$.
	It can also be computed with the mean bivariate correlation that is constant for every Bernoulli random vector $\II \in \BBB_d^{\Sigma}(\pp)$.
	For this reason, if $\II \in \BBB_d^{\Sigma}(\pp)$ has some countermonotonic pairs, it will have other pairs with higher --- possibly positive --- correlation, to keep the mean correlation equal to the constant value of the class.
	
	Under which conditions the class $\BBB_d^{\Sigma}(\pp)$, $\pp\in (0,1)^d$ supports pairwise countermonotonicity is an important issue, see \cite{lauzier2023pairwise}.
	If $\II^* \in \BBB_d^{\Sigma}(\pp)$ is pairwise countermonotonic, then every pair $(I^*_{j_1},I^*_{j_2})$ is countermonotonic, and the bivariate correlation, given the marginal means $p_{j_1}$ and $p_{j_2}$, takes the lowest possible value.
	Consequently, it is impossible to increase or decrease a bivariate correlation while keeping the mean correlation constant. 
	Thus, if a class $\BBB_d^{\Sigma}(\pp)$ admits a pairwise countermonotonic Bernoulli random vector $\II^{*}$, its distribution is the unique one in the class.
	There are two cases when $\BBB_d^{\Sigma}(\pp)$ has a unique distribution: when the lower Fréchet bound $F_L$ is a distribution, i.e., when $p_{\bullet} \leq 1$ or $p_{\bullet} \geq d-1$. 
	A random vector with distribution $F_L$ is pairwise countermonotonic. 
	Therefore, a Fréchet class $\BBB_d(\pp)$ admits a pairwise countermonotonic Bernoulli random vector if and only if $p_{\bullet} \leq 1$ or $p_{\bullet} \geq d-1$.
	We have verified that $\Sigma$-countermonotonicity and strongly Rayleigh are neither equivalent nor hierarchical with proper counterexamples. In the next section, we study the strongly Rayleigh property within the class $\mathcal{B}_d^{\Sigma}(\pp)$. We conclude this section with the following result, that shows that the NLC property is implied both from the $\Sigma$-countermonotonicity and the strongly Rayleigh properties.
	\begin{proposition}\label{prop:NLC}
		If $\II \in \BBB^{\Sigma}_d$, then $\II$ has the NLC property.
	\end{proposition}
	\begin{proof}
		If $\II \in \BBB^{\Sigma}_d$, then there exists $m \in \{0,\dots,d-1\}$ such that $f_{\II}(\ii) = 0$ for every $\ii \notin \AAA^* = \AAA_{d,m} \cup \AAA_{d,m+1}$.
		Let $\ii, \kk \in \{0,1\}^d$ and let $\ww = \ii \wedge \boldsymbol{k}$ and $\zz = \ii \vee \boldsymbol{k}$. 
		We want to prove that the condition in \eqref{eq:mtp2} is verified for every $\ii, \kk \in \{0,1\}^d$.
		Without any restriction, we can assume $i_{\bullet} \leq k_{\bullet}$.
		We have $w_{\bullet} \leq i_{\bullet} \leq k_{\bullet} \leq z_{\bullet}$, and $\ww \preceq \zz$, where $\preceq$ denotes the usual partial order in $\RR^d$, i.e. $\ww \preceq \zz$ means that $w_j \leq z_j$ for every $j \in \{1,\dots,d\}$.
		If $\ww \notin \mathcal{A}^*$ or $\zz \notin \mathcal{A}^*$, then $f(\ww)f(\zz) = 0$ and the condition in \eqref{eq:mtp2} is verified.
		Consider now  $\ww, \zz \in \mathcal{A}^*$. 
		If $w_{\bullet} = z_{\bullet}$, then $\ww \preceq \zz$ implies $\ww = \zz$, and $\ww = \ii = \kk = \zz$; thus, the condition in \eqref{eq:mtp2} is verified.
		It is verified also if $w_{\bullet} = m$ and $z_{\bullet} = m+1$, since $\ww \preceq \zz$ implies $\ww = \ii$ and $\zz = \kk$.
	\end{proof}
	
	\subsection{$\Sigma$-Countermonotonicity and the strongly Rayleigh property}
	This section discusses the strongly Rayleigh property within the convex polytope $\mathcal{B}_d^{\Sigma}(\pp)$, $\pp\in(0,1)^d$, of $\Sigma$-countermonotonic Bernoulli vectors. 
	
	Assume that $\mathcal{B}_d(\pp)$ admits a Fr\'echet lower bound and let $\II^C\in\mathcal{B}_d(\pp)$ be the Bernoulli random vector such that $F_{\II^C}$ corresponds to the lower Fr\'echet bound $F_L$. Obviously, $\II^C$ is also minimal in all the dependence orders that satisfy Condition \ref{cond:bound}. 
	Given a Fréchet class $\BBB_d(\pp)$, if $p_{\bullet} \leq 1$, the $\Sigma$-countermonotonic vector $\II^C$ has sums with support on $\{0,1\}$. 
	Its pgf is a linear polynomial $\mathcal{P}^C(\zz)=a_0 + \sum_{j=1}^d a_j z_j$, where $a_0=f(\boldsymbol{0})$ and $a_j=f(\ee_j)$, $j \in \{1,\dots,d\}$, where $\boldsymbol{0}$ is a vector of all zeros and $(\ee_1,\dots,\ee_d)$ is the canonical basis of $\RR^d$. 
	As mentioned in Remark~\ref{ex:lin}, linear polynomials are stable. Therefore, $\II^C$ is strongly~Rayleigh. The case $p_{\bullet} \geq d-1$ is similar.
	
	\begin{example}
		We start with a simple example in the special case $\sum_{j=1}^d p_j\leq 1$, where the $\Sigma$-countermonotonic vector is the random vector $\II^C$ with the lower Fréchet bound as distribution. 
		In this case, the pgf $\PPP^C$ of $\II^C$ is given by
		\begin{equation*}
			\PPP^C(z_1,\dots,z_5) = \frac{3}{20} 
			+ \frac{1}{20}(3z_1 + z_2 + 5z_3 + 6z_4 + 2z_5),
		\end{equation*}
		and it is stable. 
		Then, $\boldsymbol{I}^C$ is pairwise countermonotonic and it satisfies the strongly Rayleigh property.
	\end{example}
	The general case is more challenging as $\BBB_d^{\Sigma}(\pp)$ is a convex polytope; meaning that there are plenty $\Sigma$-countermonotonic vectors in this class. The following proposition proves that there are pmfs in $\mathcal{B}_d^{\Sigma}$ that satisfy the strongly Rayleigh property.
	
	\begin{proposition}\label{exch}
		If $\II\in \mathcal{B}_d^{\Sigma}$ is exchangeable, then it satisfies the strongly Rayleigh property.
	\end{proposition}
	
	\begin{proof}
		Since $\II\in \mathcal{B}_d^{\Sigma}$ is exchangeable, its pgf $\PPP_{\II}$ is symmetric.
		It also implies tat $S_{\II}$ is a joint mix or it has support on two consecutive points $m$ and $m+1$. 
		Consider the pgf of $S_{\II}$, that is the univariate polynomial $\PPP_S(z)=\sum_{j=0}^d s(j) z^j$, where $s$ 
		is the pmf of $S_{\II}$. 
		Case 1. $\II$ is joint mix. Then, there is $m \in \{0,\ldots, d\}$ such that $s(m)=1$. 
		It implies that the univariate polynomial $\PPP_S(z)=z^m$ has only real roots. 
		Since the pgf of $\II$ is a symmetric polynomial, by Theorem 3.8 in \cite{borcea2009negative}, $\II$ satisfies the strongly Rayleigh property. 
		Case 2. $S_{\II}$ has support on two consecutive points $m$ and $m+1$. 
		The univariate polynomial $\PPP_S(z)=s(m) z^m + s(m+1) z^{m+1}$ has only real roots. 
		By Theorem 3.8 in \cite{borcea2009negative}, since the pgf of $\II$ is a symmetric polynomial, $\II$ satisfies the strongly Rayleigh property.
	\end{proof}
	
	Let $\BBB_d(p)$ be the Fréchet class $\BBB_d(p,\dots,p)$, Proposition \ref{exch} also implies that for any $p\in(0,1)$, there is a pmf in $\mathcal{B}_d^{\Sigma}(p) =  \BBB_d^{\Sigma}\cap \BBB_d(p)$ that satisfies the strongly Rayleigh property.
	
	\begin{example}
		Consider the polynomial $\mathcal{P}(z_1,z_2,z_3)=\frac{s_2}{3}(z_1z_2+z_1z_3+z_2z_3)+z_1z_2z_3$. Let $\II\in \mathcal{B}_3(\frac{2s_2}{3}+s_3)$ be a Bernoulli random vector such that $\mathcal{P}_{\II}=\mathcal{P}.$ Then, $\II$ is exchangeable and, by Proposition \ref{exch}, it satisfies the strongly Rayleigh property. 
		If we consider the pmf in Table \ref{tab:SR}, its pgf given by $\mathcal{P}_{\II'}(z_1,z_2,z_3,z_4)=\mathcal{P}(z_1, z_2, z_3)$. 
		Then, it is strongly Rayleigh, but it is not exchangeable. 
		However, in this case $E[I'_4]=0$, and $\II'\in \mathcal{B}_4(\frac{2s_2}{3}+s_3, \frac{2s_2}{3}+s_3, \frac{2s_2}{3}+s_3,0)$ and $p_4\notin (0,1)$.  
		\begin{table}[H]
			\centering
			\begin{tabular}{c|cccc}
				\toprule
				$\ii$ & (1,1,0,0) &   (1,0,1,0) & (0,1,1,0) & (1,1,1,0) \\ 
				\midrule
				$f_{\II'}(\ii)$ & $\frac{s_2}{3}$ & $\frac{s_2}{3}$ &$\frac{s_2}{3}$ & $s_3$ \\
				\bottomrule
			\end{tabular}
			\caption{\emph{Values of $f\in \mathcal{B}_4^{\Sigma}(\frac{2s_2}{3}+s_3, \frac{2s_2}{3}+s_3, \frac{2s_2}{3}+s_3,0)$ with the strongly Rayleigh property}}
			\label{tab:SR}
		\end{table}
	\end{example}
	We prove below that if $f_{\II} \in \mathcal{B}_d^{\Sigma}(\pp)$ maximizes the entropy in $\mathcal{B}_d^{\Sigma}(\pp)$, then it satisfies the strongly Rayleigh property. 
	We recall that the  entropy $H(f_{\II})$ of a pmf $f_{\II} \in \mathcal{B}_d$ is defined as 
	\begin{equation}
		H(f_{\II}) = -\sum_{\ii \in \{0,1\}^d}
		f_{\II}(\ii) \ln(f_{\II}(\ii)).
		\label{eq:Entropy0}
	\end{equation}
	If $\II$ has pmf $f_{\II}$, we say that $\II$ has entropy $H(f_{\II})$.
	
	We prove that the distribution of $\II$ with maximal entropy pmf in $\mathcal{B}_d^{\Sigma}(\pp)$ belongs to the family of conditional Bernoulli distributions defined as follows.
	Consider a vector $\boldsymbol{K} = (K_1,\dots, K_d)$ of independent Bernoulli random variables with marginal means $\ppi = (\pi_1,\dots,\pi_d) \in (0,1)^d$.
	\begin{enumerate}
		\item
		For $m \in \{0,\dots,d\}$, let $\boldsymbol{I}^{(m)} = (I^{(m)}_1,\dots, I^{(m)}_d)$ be a Bernoulli random vector with joint pmf $f_{\ppi}^{(m)}$ given by
		\begin{equation}
			f_{\ppi}^{(m)}(\ii) = \mathbb{P} (\KK = \ii \vert K_1 + \dots K_d = m), \quad \ii \in \{0,1\}^d.
			\label{eq:ConditionningOneValue}
		\end{equation}
		
		\item
		For $m \in \{0,\dots,d-1\}$, let $\boldsymbol{I}^{(m+)} = (I^{(m+)}_1,\dots, I^{(m+)}_d)$ be a Bernoulli random vector with joint pmf $f_{\ppi}^{(m+)}$ given by
		\begin{equation}
			f_{\ppi}^{(m+)}(\ii) 
			= \mathbb{P}(\KK = \ii \vert m \leq K_1 + \dots K_d \leq m+1), \quad \ii \in \{0,1\}^d.
			\label{eq:ConditionningTwoValues}  
		\end{equation}
		
	\end{enumerate}
	
	As in \cite{chen2000general} (and references therein) or {Example 5.1 (Conditioned Bernoulli sampling)} of \cite{ghosh2017multivariate}, we provide the following definition.
	\begin{definition}
		The distribution of a Bernoulli random vector with pmf $f_{\II}$ that can be written in the form of \eqref{eq:ConditionningOneValue} or in the form of \eqref{eq:ConditionningTwoValues} is called a conditional Bernoulli distribution. A Bernoulli random vector $\II$ that follows a conditional Bernoulli distribution is referred to as a conditional Bernoulli vector.
	\end{definition}
	
	Conditional Bernoulli distributions have been extensively studied in the framework of sampling with unequal probabilities without replacement. 
	In the literature of this research field, they are also called conditional Poisson sampling; see Chapter 5 of \cite{tille2006sampling}.
	It is possible to state the following Theorem.
	Let us denote by $f^H$ the pmf of the class $\mathcal{B}_d^{\Sigma}(\pp)$ with the maximum entropy and by $\II^H=(I_1^H,\ldots, I_d^H)$ the corresponding Bernoulli random vector.
	\begin{theorem}\label{thm:maxentr}
		For any $\pp \in (0,1)^d$, there exists $\II \in \BBB_d^{\Sigma}(\pp)$ with pmf $f_{\II}$ that maximizes the entropy in the class $\BBB_d^{\Sigma}(\pp)$ that follows a conditional Bernoulli distribution and that satisfies the strongly Rayleigh property.
	\end{theorem}
	\begin{proof}
		Since the entropy $H(f_{\II})$, defined in \eqref{eq:Entropy0}, is a continuous function of $f_{\II}$ and $\BBB_d^{\Sigma}(\pp)$ is a convex polytope, $f^H$ always exists. Consider the conditional Bernoulli vectors $\II^{(m)}$ and $\II^{(m+)}$ with pmfs $f_{\ppi}^{(m)}$ in \eqref{eq:ConditionningOneValue} and $f_{\ppi}^{(m+)}$ in \eqref{eq:ConditionningTwoValues}, respectively, and let $p_j = E[I_j^{(m)}]$ and $p'_j = E[I_j^{(m+)}]$, for $j = 1,\dots,d$.
		By construction, it is clear that the Bernoulli random vector $\II^{(m)}$ is a joint mix with $\mathbb{P}(\sum_{j=1}^d I^{(m)}_j = m) = 1$ and the Bernoulli random vector $\II^{(m+)}$ is $\Sigma$-countermonotonic, i.e. the sum $\sum_{j=1}^d I^{(m+)}_j$ takes value $m$ with probability $m+1-p_{\bullet}'$ and $m+1$ with probability $p_{\bullet}'-m$, where $p_{\bullet}' = \sum_{j=1}^d p'_j$.
		Since $\KK = (K_1,\dots,K_d)$ is a vector of independent Bernoulli random variables, it satisfies the strongly Rayleigh property.
		Therefore, from Corollary 4.18 of \cite{borcea2009negative}, $\II^{(m)}$ and $\II^{(m+)}$ are strongly Rayleigh.
		It is easily proved that that the Bernoulli random vector $\II^{(m)}$ with pmf $f_{\ppi}^{(m)} \in \BBB_d^{\Sigma}(p_1,\dots,p_d)$ maximizes the entropy in the class $\BBB_d^{\Sigma}(p_1,\dots,p_d)$ (see Section 1 of \cite{chen1994weighted}).
		A straightforward generalization of this result shows that this also holds for $\II^{(m+)}$ with pmf $f_{\ppi}^{(m+)} \in \BBB_d^{\Sigma}(p'_1,\dots,p'_d)$.
		The thesis follows if we prove that for every $\pp \in (0,1)^d$, the pmf $f^H \in \BBB_d^{\Sigma}(\pp)$ with the maximum entropy is a conditional Bernoulli distribution of the form of \eqref{eq:ConditionningOneValue} or \eqref{eq:ConditionningTwoValues}, for some $\ppi \in (0,1)^d$.
		Fix $\pp \in (0,1)^d$.
		If $p_{\bullet} = m \in \{0,\dots,d\}$, Theorem 1 of \cite{chen1994weighted} ensures that there exists $\ppi = (\pi_1,\dots,\pi_d) \in (0,1)^d$ --- unique up to rescaling the ratios $\tfrac{\pi_j}{1-\pi_j}$ --- and a vector $\KK=(K_1,\dots,K_d)$ of independent random variables with marginal means $(\pi_1,\dots,\pi_d)$ such that $f^H$ is given by \eqref{eq:ConditionningOneValue}. 
		Let us now consider the case $f^H \in \BBB_d^{\Sigma}(\pp)$ with $p_{\bullet} \in (m, m+1)$, for some $m \in (0,\dots,d-1)$.
		It is straightforward to verify conditional Bernoulli distribution $f_{\ppi}^{(m+)}$ of the form of \eqref{eq:ConditionningTwoValues} can be written as follows:
		\begin{equation*}
			f_{\ppi}^{(m+)}(\ii) = f^{(m+)}_{\boldsymbol{\theta}}(\ii) = \exp \bigg\{ \sum_{j=1}^d \theta_j i_j - \psi(\theta_1,\dots,\theta_d) \bigg\}, \quad \ii \in \AAA^*,
		\end{equation*}
		where $\theta_j = \log \tfrac{\pi_j}{1-\pi_j}$, for $j \in \{1,\dots,d\}$, $\psi(\boldsymbol{\theta}) = \log \left( \sum_{\ii \in \mathcal{A}^*} \exp \{ \sum_{j=1}^d \theta_j i_j\} \right)$, and $\mathcal{A}^* = \mathcal{A}_{d,m} \cup \mathcal{A}_{d,m+1}$.
		This shows that the class of conditional Bernoulli distributions belongs to the family of exponential distributions.
		Moreover, it can be shown that, for every $m \in (0,\dots,d-1)$, the class of conditional Bernoulli distributions with pmfs of the form of \eqref{eq:ConditionningTwoValues} is a minimal regular standard exponential family, see \cite{brown1986fundamentals}.
		Therefore, by Theorem 3.6 of \cite{brown1986fundamentals}, the map between the natural parameters $\theta_1,\dots,\theta_d$ and the marginal means $p_1,\dots,p_d$ is one to one. 
		This implies that, there exists a unique $\ppi^* = (\pi^*_1,\dots,\pi^*_d) \in (0,1)^d$ such that $f^H = f_{\ppi}^{(m+)}$.
	\end{proof}

	Since every Fréchet class supports a $\Sigma$-countermonotonic vector, Theorem \ref{thm:maxentr} implies that every Fréchet class admits a Bernoulli random vector that satisfies both the strongly Rayleigh and the $\Sigma$-countermonotonicity property.
	\begin{remark} 
		\label{rmk:ConditionalBernoulli}
		We observe the following facts from the proof of Theorem \ref{thm:maxentr}. The Bernoulli random vector  $\boldsymbol{I}^{(m)}$  satisfies the strongly Rayleigh and the joint mixability properties. Thus, it is $\Sigma$-countermonotonic.
		Moreover, $\II^{(m)}$ has the maximal entropy pmf in the class $\BBB_d^{\Sigma}(\pp)$, where $\pp = (p_1,\dots,p_d)$ and $p_j = E[I^{(m)}_j]$, for $j \in \{1,\dots,d\}$.
		The vector, $\boldsymbol{I}^{(m+)}$ satisfies the strongly Rayleigh property and it is the $\Sigma$-countermonotonic Bernoulli random vector with the maximal entropy pmf in the class $\BBB_d^{\Sigma}(\pp')$, where $\pp' = (p'_1,\dots,p'_d)$ and $p'_j = E[I^{(m+)}_j]$, for $j \in \{1,\dots,d\}$. 
		
	\end{remark}
	We investigate below the construction of $f^H$ as a conditional Bernoulli distribution.
	We explicitly derive the expression of a conditional Bernoulli pmf in the two cases specified by \eqref{eq:ConditionningOneValue} and \eqref{eq:ConditionningTwoValues}.
	Let $\boldsymbol{K} = (K_1,\dots,K_d)$ be a vector of $d$ independent Bernoulli random variables with $K_j \sim Bern(\pi_j)$ with $0 < \pi_j < 1$ for $j \in \{1,\dots,d\}$.
	The joint pmf $g$ of the random vector $\KK$ is given by
	\begin{equation*}
		g(\ii; \ppi) 
		=
		\prod_{j=1}^d
		\pi_j^{i_j} (1-\pi_j)^{(1-i_j)},
		\quad
		\text{for } \ii \in \{0,1\}^d.
	\end{equation*}
	Moreover, we denote by $\gamma$ the univariate pmf of $\sum_{j=1}^d K_j$, that is
	\begin{equation*}
		\gamma(m;\ppi)
		= \sum_{\ii \in \mathcal{A}_{d,m}}
		g(\ii; \ppi),
		\quad
		\text{for } m \in \{0,\dots,d\}.
	\end{equation*}
	
	Fix $m \in \{0,\dots,d\}$.
	From \eqref{eq:ConditionningOneValue}, the values of the conditional Bernoulli distribution $f_{\ppi}^{(m)}$ of a conditional Bernoulli vector $\II^{(m)}$ with the joint mixability property is given by 
	\begin{align}
		f_{\ppi}^{(m)}(\ii) 
		= \begin{cases}
			\frac{1}{\gamma(m;\ppi)} g(\ii;\ppi),
			&  \ii \in \mathcal{A}_{d,m}  \\
			0, & \ii \notin \mathcal{A}_{d,m}.
		\end{cases}
		\label{eq:PmfofVectorI}
	\end{align}
	
	The conditional Bernoulli pmf $f_{\ppi}^{(m)}$ that we obtain in \eqref{eq:PmfofVectorI} belongs to a class $\BBB_d^{\Sigma}(\pp)$, for some $\pp = (p_1,\dots,p_d)$ such that $p_{\bullet} = m$.
	
	Now, fix two consecutive integers, $m \in \{0,\dots,d-1\}$ and $m+1$. 
	From \eqref{eq:ConditionningTwoValues}, the values of the conditional Bernoulli distribution $f_{\ppi}^{(m+)}$ of a conditional Bernoulli vector $\II^{(m+)}$ with the $\Sigma$-countermonotonicity property is given by 
	\begin{align}
		f_{\ppi}^{(m+)}(\ii) 
		= \begin{cases}
			\frac{1}{\gamma(m;\ppi) + \gamma(m+1;\ppi)} g(\ii;\ppi),
			&  \ii \in \mathcal{A}_{d,m} \cup \mathcal{A}_{d,m+1}  \\
			0, & \ii \notin (\mathcal{A}_{d,m} \cup \mathcal{A}_{d,m+1})
		\end{cases}.
		\label{eq:PmfofVectorI+}
	\end{align}
	The conditional Bernoulli pmf $f_{\ppi}^{(m+)}$ that we obtain in \eqref{eq:PmfofVectorI+} is a member of a class $\BBB_d^{\Sigma}(\pp)$, for some $\pp = (p_1,\dots,p_d)$ such that $p_{\bullet} \in (m,m+1)$.
	
	By the proof of Theorem \ref{thm:maxentr}, given an integer $m$ or a couple of consecutive integers $m$ and $m+1$, the map between the marginal means $\ppi$ of the independence vector $\KK$ and the marginal means $\pp$ of the conditional Bernoulli distribution is one-to-one (up to rescaling when we condition on a single value $m$).
	Given $\ppi \in (0,1)^d$, the values $\pp$ can be computed by evaluating the marginal means of the conditional Bernoulli pmfs in \eqref{eq:PmfofVectorI} or \eqref{eq:PmfofVectorI+}.
	
	However, for a fixed class $\BBB_d^{\Sigma}(\pp)$, finding the values of the marginal means $\ppi^*$ that generate the conditional Bernoulli that belongs to this class is not straightforward. 
	The authors of \cite{chen1994weighted} proposed a geometrically convergent algorithm, while here we present the solution of this problem as a solution of an optimization problem.
	Given any $\pp \in (0,1)^d$, let us define a conditional Bernoulli pmf $f^{\text{CB}}_{\ppi}$ by
	\begin{equation*}
		f^{\text{CB}}_{\ppi} =
		\begin{cases}
			f^{(m)}_{\ppi},  &\text{if } p_{\bullet} = m \in \{0,\dots,d\}
			\\
			f^{(m+)}_{\ppi},  &\text{if } p_{\bullet} \in (m,m+1) \text{ for } m \in \{0,\dots,d-1\}
		\end{cases}.
	\end{equation*}
	We define the following maximization problem,
	\begin{equation}
		\ppi^*=\text{arg}\max_{\ppi \in(0,1)^d} H(f_{\ppi}^{\text{CB}}),
		\label{eq:MaxProblem}
	\end{equation}
	subject to the constraints
	\begin{equation}
		\sum_{\ii \in \{0,1\}^d} x_j f_{\ppi}^{\text{CB}}(\ii) = p_j,
		\quad j = 1,\dots,d.
		\label{eq:constraints}
	\end{equation}
	Notice that for $f_{\ppi}^{\text{CB}} \in \mathcal{B}_d^{\Sigma}(\pp)$, the expression of the entropy in \eqref{eq:Entropy0} becomes
	\begin{equation} \label{eq:Entropy}
		H(f_{\ppi}^{\text{CB}})=-\sum_{\ii \in \mathcal{A}^*}
		f_{\ppi}^{\text{CB}}(\ii) \ln(f_{\ppi}^{\text{CB}}(\ii)).
	\end{equation}
	where $\AAA^* = \AAA_{d,m}$, if $p_{\bullet} = m$, or $\AAA^* = \AAA_{d,m} \cup \AAA_{d,m+1}$, if $p_{\bullet} \in (m,m+1)$.
	Therefore, given a Fréchet class $\BBB_d(\pp)$, a Bernoulli random vector that satisfies both the strongly Rayleigh property and the $\Sigma$-countermonotonicity property is the maximum entropy Bernoulli random vector $\II^H$ with pmf given by $f^H = f^{\text{CB}}_{\ppi^*}$.
	
	Finally, since Proposition \ref{prop:polcx} proves that $\mathcal{B}_d^{\Sigma}(\pp)$ is a convex polytope with $n_d^{\Sigma}(\pp)$ extremal points, we have also the following representation:
	\begin{equation}
		f^H(\ii)
		=
		\sum_{k=1}^{n_d^{\Sigma}(\pp)}
		\alpha_k f_{\RRR_k}(\ii),
		\quad
		\ii \in \{0,1\}^d,
		\label{eq:ConvexCombi}
	\end{equation}
	where $f_{\RRR_k},\, k=1,\ldots n_d^{\Sigma}(\pp)$, are the extremal points  of $\mathcal{B}_d^{\Sigma}(\pp)$ and $\sum_{k=1}^{n_d^{\Sigma}(\pp)} \alpha_k = 1$ with $ \alpha_k \geq 0$ for every $k = 1,\ldots, n_d^{\Sigma}(\pp)$.
	We illustrate this methodology in the following example.
	
	\begin{example} 
		Consider the convex polytope $\BBB_4^{\Sigma}(\pp)$ with $\pp = \left(\tfrac{7}{20},\tfrac{9}{20},\tfrac{10}{20},\tfrac{14}{20}\right)$. 
		Given that $p_{\bullet}=2$, the class $\BBB_4^{\Sigma}(\pp)$ is the set of joint pmfs of the vectors of four Bernoulli random variables satisfying the joint mixability property, i.e., whose sum is equal to $2$ with probability one.
		This convex polytope has three extremal points, denoted by $f_{\RRR_{1}}$, $f_{\RRR_{2}}$, and $f_{\RRR_{3}}$ (as defined in \eqref{eq:PmfExtremalPoints}), that are provided in Table~\ref{tab:ExtremalPtsJointMix}.
		The pmf $f^{H}$ of the Bernoulli random vector $\II^{H}$ satisfying both the joint mixability property and the strongly Rayleigh property is obtained with \eqref{eq:PmfofVectorI}; it maximizes \eqref{eq:Entropy} with the constraints in \eqref{eq:constraints}. The values of $f^{H}$ are provided in Table~\ref{tab:ExtremalPtsJointMix}.
		Any pmf $f_{\II} \in \BBB_4^{\Sigma}(\pp)$ admits the following representation:
		\begin{equation}
			f_{\II}(\xx)
			=
			\alpha_1 f_{\RRR_{1}}(\ii)
			+
			\alpha_2 f_{\RRR_{2}}(\ii)
			+
			\alpha_3 f_{\RRR_{3}}(\ii),
			\quad
			\ii \in \mathcal{A}_{d,m},
			\label{eq:ConvexCombi_example}
		\end{equation}
		where $\alpha_1, \alpha_2, \alpha_3 \geq 0$ and $\alpha_1 + \alpha_2 + \alpha_3 = 1$. 
		For $f^{H}$, the values of the coefficients in \eqref{eq:ConvexCombi_example} are $(\alpha_1, \alpha_2, \alpha_3) = (0.4235, 0.3090, 0.2675)$.
		
		\begin{table}[H]
			\centering
			\begin{tabular}{c|cccccc|c}
				\toprule
				$\ii$ & (1,1,0,0) & (1,0,1,0) & (0,1,1,0) & (1,0,0,1) & (0,1,0,1) & (0,0,1,1) & H \\ 
				\midrule
				$f_{\RRR_{1}}(\ii)$ & 0 & 0 & $\tfrac{3}{10}$ & $\tfrac{7}{20}$ & $\tfrac{3}{20}$ & $\tfrac{1}{5}$ & $1.3351$ \\
				$f_{\RRR_{2}}(\ii)$ & 0 & $\tfrac{3}{10}$ & 0 & $\tfrac{1}{20}$ & $\tfrac{9}{20}$ & $\tfrac{1}{5}$ & $1.1922$ \\
				$f_{\RRR_{3}}(\ii)$ & $\tfrac{3}{10}$ & 0 & 0 & $\tfrac{1}{20}$ & $\tfrac{3}{20}$ & $\tfrac{1}{2}$ & $1.1421$  \\
				$f^{H}(\ii)$ & 
				$0.0802$ & $0.0927$ 
				& $ 0.1771$ & $0.1271$ 
				& $ 0.2427$ & $0.2803$ & $1.6917$ \\
				\bottomrule
			\end{tabular}
			\caption{\emph{Values of the 3 extremal points of $\BBB_4^{\Sigma}(\pp)$, where $\pp = \left(\tfrac{7}{20},\tfrac{9}{20},\tfrac{10}{20},\tfrac{14}{20}\right)$, and values of $f^{H}$.}}
			\label{tab:ExtremalPtsJointMix}
		\end{table}
		
	\end{example}
	
	The following example considers the case with $p_1= \dots = p_d = p$ leading to an explicit form of the pmf maximizing the entropy in \eqref{eq:Entropy0}, that is, the pmf of an exchangeable Bernoulli random vector in the class $\mathcal{B}_d^{\Sigma}(p)$, already proven to satisfy the strongly Rayleigh property in Proposition \ref{exch}.
	
	\begin{example} \label{ex:CaseEqualp}
		Let $\boldsymbol{K}$ be a vector of independent Bernoulli random variables with $K_i\sim Bern (\pi)$, for $i=1,\ldots, d$. To recover a pmf satisfying both the strongly Rayleigh and the joint mixability properties, fix $m \in \{0,\dots,d\}$ and for $\ii \in \mathcal{A}_{d,m}$, we have
		\begin{equation*}
			\begin{split}
				f_{\pi}^{(m)}(\ii) 
				&= \mathbb{P}(K_1 = i_1, \dots, K_d = i_d \vert K_1 + \dots + K_d = m)\\
				&=\frac{\pi^m(1-\pi)^{d-m}}{\sum_{\ii\in\mathcal{A}_{d,m}}\pi^m(1-\pi)^{d-m}}=\frac{\pi^m(1-\pi)^{d-m}}{\binom{d}{m}\pi^m(1-\pi)^{d-m}},
			\end{split}
		\end{equation*}
		which becomes
		\begin{equation} \label{eq:pmf_fm_example}
			f_{\pi}^{(m)}(\ii) =\frac{1}{\binom{d}{m}}.
		\end{equation}
		As discussed in Remark \ref{rmk:ConditionalBernoulli}, $f_{\pi}^{(m)}$ is the maximum entropy pmf in  $\BBB_d^{\Sigma}(p)$, where $p = \frac{m}{d}$. Notice that $p$ and $f_{\pi}^{(m)}$ in \eqref{eq:pmf_fm_example} do not depend on the initial probability $\pi$.
		Let us now consider the $\Sigma$-countermonotonic case. 
		Fix $m \in \{0,\dots,d-1\}$, then from \eqref{eq:ConditionningTwoValues} we have
		\begin{equation*}
			f_{\pi}^{(m+)}(\ii) 
			=
			\begin{cases}
				\mathbb{P}(K_1 = i_1, \dots, K_d = i_d \vert m \leq K_1 + \dots + K_d \leq m+1),
				& \ii \in \mathcal{A}_{d,m} \cup \mathcal{A}_{d,m+1} \\
				0, & \text{otherwise}
			\end{cases}.
		\end{equation*}
		In this case, the sum $\sum_{j=1}^d K_j$ takes value $m$ with probability $\gamma(m;\pi) = \binom{d}{m} \pi^m (1-\pi)^{d-m}$ and value $m+1$ with probability $\gamma(m+1;\pi) = \binom{d}{m+1} \pi^{m+1} (1-\pi)^{d-m-1}$.
		Therefore, we have
		\begin{equation*}
			f_{\pi}^{(m+)}(\ii) 
			=
			\begin{cases}
				\frac{\pi^m (1-\pi)^{d-m}}{\gamma(m;\pi) + \gamma(m+1;\pi)},
				& \ii \in \mathcal{A}_{d,m} \\
				\frac{\pi^{m+1} (1-\pi)^{d-m-1}}{\gamma(m;\pi) + \gamma(m+1;\pi)},
				& \ii \in \mathcal{A}_{d,m+1} \\
				0, & \ii \notin (\mathcal{A}_{d,m} \cup \mathcal{A}_{d,m+1})
			\end{cases},
		\end{equation*}
		which becomes
		\begin{equation*}
			f_{\pi}^{(m+)}(\ii) 
			=
			\begin{cases}
				\frac{1-\pi}{\eta},
				& \ii \in \mathcal{A}_{d,m} \\
				\frac{\pi}{\eta},
				& \ii \in \mathcal{A}_{d,m+1} \\
				0, & \ii \notin (\mathcal{A}_{d,m} \cup \mathcal{A}_{d,m+1})
			\end{cases},
		\end{equation*}
		where $\eta = \binom{d}{m}(1-\pi) + \binom{d}{m+1} \pi$.
		This is the maximum entropy pmf in the class $\BBB_d^{\Sigma}(p)$, where 
		\begin{equation*}
			p = \frac{m+1}{d} \frac{m(1-\pi) + (d-m)\pi}{(m+1)(1-\pi) + (d-m)\pi}.
		\end{equation*}
		In this case, the final marginal mean $p$ depends on the starting probability $\pi$.
		The values of  the pmf $s_{\pi}^{(m+)}$ of $S^{(m+)} = \sum_{j=1}^d I_j^{(m+)}$, where $\II^{(m+)} = (I_1^{(m+)}, \dots, I_d^{(m+)})$ is a Bernoulli vector with pmf $f_{\pi}^{(m+)}$, is given by
		\begin{equation*}
			s_{\pi}^{(m+)}(m) = \frac{\binom{d}{m}(1-\pi)}{\binom{d}{m}(1-\pi) + \binom{d}{m+1}\pi},
			\quad \text{and} \quad
			s_{\pi}^{(m+)}(m+1) = \frac{\binom{d}{m+1}\pi}{\binom{d}{m}(1-\pi) + \binom{d}{m+1}\pi},
		\end{equation*}
		and zero elsewhere.
		The parameter $\pi$ can be interpreted as the parameter that moves the expected value of $S^{(m+)}$ between $m$ and $m+1$.
		Notice that the pmf $s_{\pi}^{(m+)}$ coincides with $s_{p_{\bullet}}$ in \eqref{eq:ExtrPoints_D}, with $p_{\bullet} = pd$.
		
	\end{example}
	
	The following Corollary characterizes the pgf associated to the maximal entropy pmf $f^H$ in terms of its pgf in  \eqref{pgfS} and \eqref{pgfS2}.
	\begin{corollary} \label{cor:pgf_fH}
		Let us consider the maximal entropy pmf $f^H$ in $\BBB_d^{\Sigma}(\pp)$.
		Then, there exists $\aa = (a_1,\ldots, a_d)$, with $a_j >0$, such that the pgf is given by
		\begin{equation} \label{eq:pgf_maxentr}
			\mathcal{P}^H(\zz)\propto\sum_{\ii \in \mathcal{A}^*}\aa^{\ii}\zz^{\ii},
		\end{equation}
		where $\mathcal{A}^* = \mathcal{A}_{d,p_{\bullet}}$ if $p_{\bullet}$ is integer, or $\mathcal{A}^* = \mathcal{A}_{d,m} \cup \mathcal{A}_{d,m+1}$ if $p_{\bullet} \in (m,m+1)$, for $m \in (0,\dots,d-1)$.
	\end{corollary}
	\begin{proof}
		From Theorem \ref{thm:maxentr}, there exists a vector $\KK = (K_1,\dots,K_d)$ of $d$ independent Bernoulli random variables with marginal means $\ppi = (\pi_1,\dots,\pi_d)$ such that $f^H$ is given by \eqref{eq:ConditionningOneValue} or by \eqref{eq:ConditionningTwoValues}.
		The pgf of $\KK$ is given by
		\begin{equation} \label{eq:easteregg}
			\mathcal{P}^{\perp}(\zz) =  \prod_{j=1}^d (\pi_j z_j + 1 - \pi_j) = \prod_{j=1}^d (1-\pi_j) \sum_{\ii \in \{0,1\}^d} \aa^{\ii}\zz^{\ii},
		\end{equation}
		where $a_j = \frac{\pi_j}{1-\pi_j}$, for $j \in \{1,\dots,d\}$, and $\aa = (a_1,\dots,a_d)$.
		Since $f^H$ is the pmf of $\KK$ conditioned to taking values on the set $\mathcal{A}^*$, we have that its pgf $\PPP^H$ is proportional to the terms of $\PPP^{\perp}$ of order $m$ if $p_{\bullet} = m$, or of orders $m$ and $m+1$ if $p_{\bullet} \in (m,m+1)$. $\PPP^H$ is normalized such that the sum of the coefficients is one.
		Therefore, we have \eqref{eq:pgf_maxentr}, with $a_j = \frac{\pi_j}{1-\pi_j}$, for $j \in \{1,\dots,d\}$.
	\end{proof}
	
	\begin{example}
		Let $a_1, a_2,a_3 > 0$. The expression of $\mathcal{P}^{\perp}$ defined in \eqref{eq:easteregg} is given by
		\begin{equation*}
			\begin{split}
				\mathcal{P}^{\perp}(z_1,z_2,z_3) &\propto  (a_1z_1 + 1) (a_2z_2 + 1) (a_3z_3 + 1)
				\\
				&=
				1 + a_1z_1 + a_2z_2 + a_1a_2z_1z_2 + a_3z_3 + a_1a_3z_1z_3 + a_2a_3z_2z_3 + a_1a_2a_3z_1z_2z_3
				\\
				&= \sum_{\ii \in \{0,1\}^3} \aa^{\ii}\zz^{\ii}.
			\end{split}
		\end{equation*}
		Consider the Fréchet class $\mathcal{B}_3(p_1,p_2,p_3)$ with $p_1+p_2+p_3 \in (1,2)$.
		By Theorem \ref{thm:maxentr}, the maximum entropy pmf $f^H \in \mathcal{B}_3^{\Sigma}(p_1,p_2,p_3)$ is the conditional Bernoulli pmf $f^{(1+)}_{\ppi^*}$, with $\ppi^*$ the solution to the optimization problem in \eqref{eq:MaxProblem} subject to the three constraints specified by \eqref{eq:constraints}.
		Let $a_j^* = \frac{\pi^*_j}{1-\pi^*_j}$, for $j=1,2,3$. Then, the pgf associated to $f^H$ is given by
		\begin{equation*}
			\mathcal{P}^H(\zz) \propto a_1^* z_1 + a_2^* z_2 + a_1^* a_2^* z_1z_2 + a_3^* z_3 + a_1^* a_3^* z_1z_3 + a_2^* a_3^* z_2z_3
			= \sum_{\ii \in \mathcal{A}^*}(\aa^*)^{\ii}\zz^{\ii},
		\end{equation*}
		where $\AAA^* = \AAA_{3,1} \cup \AAA_{3,2}$.
	\end{example}
	
	\begin{remark}
		The pgf $\PPP^H$ associated to the maximum entropy pmf $f^H$ is such that
		\begin{equation*}
			\PPP^H \bigg(\frac{z_1}{a_1},\dots,\frac{z_d}{a_d} \bigg) \propto \sum_{\ii \in \AAA^*} \zz^{\ii} = \EEE_{p_{\bullet}}
		\end{equation*}
		where $\EEE_{p_{\bullet}} = \EEE_{d,p_{\bullet}}$ if $p_{\bullet}$ is integer, otherwise $\EEE_{p_{\bullet}} = \EEE_{d,m} + \EEE_{d,m+1}$ if $p_{\bullet} \in (m,m+1)$, and where $\EEE_{d,m}$ is the elementary symmetric polynomial of order $m$.
		This is another way to show that $f^H$ satisfies the strongly Rayleigh property. As recalled in Remark \ref{ex:lin}, $\EEE_{d,m}$ is stable for every $m \in \{0,\dots,d\}$. Moreover, by Theorem \ref{thm:LinearComb_StrongRayleigh}, $\EEE_{d,m}+\EEE_{d,m+1}$ is stable because $z_{d+1}\EEE_{d,m}+\EEE_{d,m+1} = \EEE_{d+1,m+1}$ is stable.
		Therefore, $\EEE_{p_{\bullet}}$ is a stable polynomial for every $p_{\bullet} \in (0,d)$, and from Point \ref{point:scaling} of Lemma \ref{lemma:StabilityPreservers}, $\PPP^H$ is stable.
	\end{remark}
	In general, the maximal entropy pmf in $\mathcal{B}_d^{\Sigma}(\pp)$ is not the unique pmf in $\mathcal{B}_d^{\Sigma}(\pp)$ with the strongly Rayleigh property.
	Beyond Theorem \ref{thm:maxentr}, there are other approaches to find Bernoulli random vectors that enjoy both the strongly Rayleigh and the $\Sigma$-countermonotonicity properties.
	The following example shows that it is possible to find strongly Rayleigh pmfs in $\BBB_d^{\Sigma}(\pp)$ by slightly modifying $f^H$, using its representation as convex combination of extremal points in \eqref{eq:ConvexCombi}.

	\begin{example}
		Consider the class $\mathcal{B}_3^{\Sigma}(\tfrac{2}{5})$ for which the extremal points $f_{\RRR_k}$, $k=1,2,3$ are reported in Table \ref{ex:contSR+}. The maximum entropy pmf $f^H \in \mathcal{B}_3^{\Sigma}(\tfrac{2}{5})$ is the exchangeable pmf in Table \ref{ex:contSR+}, and it is such that
		\begin{equation} \label{eq:IndepConvexComb}
			f^H(\ii) = \frac{1}{3} f_{\RRR_1}(\ii) + \frac{1}{3} f_{\RRR_2}(\ii) + \frac{1}{3} f_{\RRR_3}(\ii),
		\end{equation}
		for all $\ii \in \{0,1\}^3$.
		In order to find another pmf in the class $\mathcal{B}_3^{\Sigma}(\tfrac{2}{5})$ that satisfies the strongly Rayleigh property, we slightly perturb the coefficient of the convex linear combination in \eqref{eq:IndepConvexComb}.
		Consider the pmf, $\tilde{f} \in \mathcal{B}_3^{\Sigma}(\tfrac{2}{5})$ with $\tilde{f}(\ii)=\sum_{k=1}^3\alpha_{k}f_{\RRR_k}(\ii)$, for $\ii \in \{0,1\}^3$, and  $\alpha_1= 0.333$, $\alpha_2=0.335$, $\alpha_3=0.332$ given in Table \ref{ex:contSR+}.
		The pgf associated to $\tilde{f}$ is $\tilde{\mathcal{P}}(\zz)=0.2664z_1+0.267z_2+0.2666z_3+0.0666z_1z_2+0.067z_1z_3+0.0664 z_2z_3$  satisfies the conditions for the strongly Rayleigh property in  \eqref{eq:CondStrongRayleigh} for all $1 \leq j_1, j_2 \leq 3$ and $\xx \in \RR^3$.
		Therefore, $\tilde{f}$ satisfies the strongly Rayleigh and $\Sigma$-countermonotonicity properties, but it does not have the maximum entropy in the class $\mathcal{B}_3^{\Sigma}(\tfrac{2}{5})$.
		\begin{table}[bt]
			\centering
			\begin{tabular}{c|ccc|cc}
				\toprule
				$\ii$ & $f_{\RRR_1}(\ii)$ & $f_{\RRR_2}(\ii)$ & $f_{\RRR_3}(\ii)$ & $f^H(\ii)$ & $\tilde{f}(\ii)$ \\
				\midrule
				(0,0,0) & 0 & 0 & 0 & 0 & 0 \\ 
				(1,0,0) & 0.2 & 0.2 & 0.4 & $\tfrac{4}{15}$ & 0.2664 \\ 
				(0,1,0) & 0.2 & 0.4 & 0.2 & $\tfrac{4}{15}$ & 0.267 \\ 
				(1,1,0) & 0.2 & 0 & 0 & $\tfrac{1}{15}$ & 0.0666 \\ 
				(0,0,1) & 0.4 & 0.2 & 0.2 & $\tfrac{4}{15}$ & 0.2666 \\ 
				(1,0,1) & 0 & 0.2 & 0 & $\tfrac{1}{15}$ & 0.067 \\ 
				(0,1,1) & 0 & 0 & 0.2 & $\tfrac{1}{15}$ & 0.0664 \\ 
				(1,1,1) & 0 & 0 & 0 & 0 & 0 \\
				\midrule
				Entropy & 1.33217904 & 1.33217904 & 1.33217904 & 1.599014712 & 1.599012963 \\
				\bottomrule
			\end{tabular}
			\caption{\emph{$f_{\RRR_1}$, $f_{\RRR_2}$, and $f_{\RRR_3}$ are the extremal points of $\mathcal{B}_3^{\Sigma}(\tfrac{2}{5})$, $f^H$ is the maximum entropy pmf in $\mathcal{B}_3^{\Sigma}(\tfrac{2}{5})$, and $\tilde{f} \in \mathcal{B}_3^{\Sigma}(\tfrac{2}{5})$ is strongly Rayleigh and $\Sigma$-countermonotonic.}}
			\label{ex:contSR+}
		\end{table}
	\end{example}
	
	A second approach is based on the polarization of a stable polynomial, see Theorem 4.6 and Corollary 4.7 of \cite{borcea2009negative}. 
	However, this construction requires some conditions on the Fréchet class.
	
	\begin{proposition} \label{prop:Polarization}
		Let $\pp \in (0,1)^d$.
		Suppose that there exists a partition $(\Lambda_1,\dots,\Lambda_D)$ of $\{1,\dots,d\}$ with cardinality $(\lambda_1,\dots,\lambda_D)$ such that, for all $h \in \{1,\dots,D\}$, we have that $p_j = \tilde{p}_h \leq \frac{1}{\lambda_h}$ for every $j \in \Lambda_h$.
		For every $h \in \{1,\dots,D\}$, let $\KK_h \in \BBB_{\lambda_h}(\frac{1}{\lambda_h})$ be a random vector with distribution the lower Fréchet bound of $\BBB_{\lambda_h}(\frac{1}{\lambda_h})$.
		Let $\II^H$ be the maximum entropy $D$-dimensional Bernoulli random vector of $\BBB_D^{\Sigma}(\lambda_1\tilde{p}_1, \dots, \lambda_D \tilde{p}_D)$ and assume that $\KK_1,\dots,\KK_D$, and $\II^H$ are independent. 
		Then, the random vector $\boldsymbol{J} = (\JJ_1,\dots,\JJ_D)$, where, for every $h \in \{1,\dots,D\}$,
		\begin{equation*}
			\JJ_h = I_h^H \KK_h,
		\end{equation*}
		has distribution in the class $\BBB_d(\pp)$ and satisfies both the strongly Rayleigh and the $\Sigma$-countermonotonicity properties.
	\end{proposition}
	\begin{proof}
		The Fréchet class $\BBB_{\lambda_h}(\frac{1}{\lambda_h})$ always admits the lower Fréchet bound as a distribution because the sum of the marginal means is equal to 1.
		Since $\mathbb{P}(\sum_{l=1}^{\lambda_h} K_{h,l} = 1) = 1$, we have $\mathbb{P}(\sum_{l=1}^{\lambda_h} J_{h,l} = I_h^H) = 1$ for every $h \in \{1,\dots,D\}$, and 
		$\PP(\sum_{j=1}^d J_j = \sum_{h=1}^D I_h^H) = 1$.
		Therefore, since $\II^H$ is $\Sigma$-countermonotonic, also $\JJ$ is $\Sigma$-countermonotonic.
		We now prove that $\JJ$ satisfies the strongly Rayleigh property.
		For every $h \in \{1,\dots, D\}$, the pgf of $\KK_h$ is
		\begin{equation*}
			\PPP_{\KK_h} (z_1,\dots,z_{\lambda_h})= \frac{1}{\lambda_h} \EEE_{\lambda_h,1}(z_1,\dots,z_{\lambda_h}),
		\end{equation*}
		and, by Theorem \ref{thm:maxentr}, the pgf $\PPP^H(z_1,\dots,z_D)$ of $\II^H$ is stable.
		Due to the independence assumption, it is possible to show that the pgf of $\JJ$ can be written as
		\begin{equation*}
			\PPP_{\JJ}(z_1,\dots,z_d) = \PPP^H(\PPP_{\KK_1}(z_j, j \in \Lambda_1),\dots,\PPP_{\KK_D}(z_j, j \in \Lambda_D)). 
		\end{equation*}
		Thus, by Corollary 4.7 of \cite{borcea2009negative}, since $\PPP^H$ is stable, $\PPP_{\JJ}$ is also stable and $\JJ$ satisfies the strongly Rayleigh property.
	\end{proof}
	We notice that for every Fréchet class $\BBB_d(p)$ with $p \leq \frac{1}{2}$ and $d$ even, the construction in Proposition \ref{prop:Polarization} is always feasible.
	We present the details of this construction in Example \ref{ex:Polarization}.
	\begin{example} \label{ex:Polarization}
		Let $\boldsymbol{I}^H = (I^H_1,\dots,I^H_D)$ be a Bernoulli random vector with the maximum entropy pmf $f^H \in \BBB_D^{\Sigma}(p')$, with $p' \in (0,1)$, that satisfies the $\Sigma$-countermonotonicity property and, by Theorem \ref{thm:maxentr}, the strongly Rayleigh property. 
		Let $\KK_h=(K_{h,1},K_{h,2})$ be a countermonotonic Bernoulli random vector with marginal means $\frac{1}{2}$, for every $h\in \{1,\dots,D\}$. 
		Assume that $\KK_1, \dots, \KK_D$, and $\II^H$ are independent. Define $\JJ = (\JJ_1,\dots,\JJ_D)$, where
		\begin{align*}
			\JJ_h &= I^H_h \KK_h, \quad h \in \{1,\dots,D\}.
		\end{align*}
		We have $\JJ \in \BBB_{d}(p)$, where $d=2D$ and $p = \frac{p'}{2}$, and the joint pgf of $\JJ$ is given by 
		\begin{align*}
			\mathcal{P}_{\JJ}(z_1,\dots,z_{d})
			& =
			E \bigg[ \prod_{j=1}^{d} z_j^{J_j} \bigg] =
			E \bigg[ \prod_{h=1}^D z_{2h-1}^{I^H_h K_{h,1}}  z_{2h}^{I^H_h K_{h,2}} \bigg] 
		\end{align*}
		Conditioning on $\boldsymbol{J}$ and given the assumption of independence, we obtain
		\begin{align*}
			\mathcal{P}_{\JJ}(z_1,\dots,z_d)
			& = 
			E \bigg(
			E \bigg[ \prod_{h=1}^D z_{2h-1}^{I^H_h K_{h,1}}  z_{2h}^{I^H_hK_{h,2}} 
			\bigg\vert
			(I^H_1,\dots,I^H_D)
			\bigg] \bigg)
			= E \bigg(
			\prod_{h=1}^D E \bigg[  z_{2h-1}^{K_{h,1}}z_{2h}^{K_{h,2}} \bigg] ^{I_h^H}
			\bigg)
			\\ 
			& =
			\mathcal{P}^H
			(\mathcal{P}_{\KK_1}(z_1,z_2),\dots,\mathcal{P}_{\KK_D}(z_{d-1},z_d)),
		\end{align*}
		From Example \ref{ex:CaseEqualp}, we can derive the expression for the pgf of $\II^H$ in the following two cases.
		If $p'D$ is integer, we have
		\begin{equation*}
			\PPP^H(z_1,\dots,z_{D}) = \frac{1}{\binom{D}{p'D}} \EEE_{d,p'D}(z_1,\dots,z_{D}),
		\end{equation*}
		while if $p'D \in (m,m+1)$, for some $m \in \{0,\dots,D-1\}$, the expression of $\PPP^H$ is
		\begin{equation*}
			\PPP^H(z_1,\dots,z_{D}) = 
			\frac{1-\pi}{\binom{D}{m}(1-\pi) + \binom{D}{m+1}\pi} \EEE_{D,m}(z_1,\dots,z_{D})
			+
			\frac{\pi}{\binom{D}{m}(1-\pi) + \binom{D}{m+1}\pi} \EEE_{D,m+1}(z_1,\dots,z_{D}),
		\end{equation*}
		where $\pi \in (0,1)$ is the value of the marginal means of the random vector $\KK$ with independent marginals such that $f^H$ is given by \eqref{eq:ConditionningTwoValues}.
		Therefore, since $p'D=pd$, we find
		\begin{align*}
			\mathcal{P}_{\JJ}(z_1,\dots,z_d)
			=
			\frac{1}{\binom{d/2}{pd}} \EEE_{d/2,pd} \bigg(
			\frac{1}{2}(z_1+z_2),\dots,\frac{1}{2}(z_{d-1}+z_d)
			\bigg),
		\end{align*}
		if $pd$ is integer, and
		\begin{align*}
			\mathcal{P}_{\JJ}(z_1,\dots,z_d)
			=
			&\frac{1-\pi}{\binom{d/2}{m}(1-\pi) + \binom{d/2}{m+1}\pi} 
			\EEE_{d/2,m}\bigg( \frac{1}{2}(z_1+z_2),\dots,\frac{1}{2}(z_{d-1}+z_d) \bigg)
			\\
			+
			&\frac{\pi}{\binom{d/2}{m}(1-\pi) + \binom{d/2}{m+1}\pi} 
			\EEE_{d/2,m+1}\bigg( \frac{1}{2}(z_1+z_2),\dots,\frac{1}{2}(z_{d-1}+z_d) \bigg),
		\end{align*}
		if $pd \in (m,m+1)$.
		In both cases, by Proposition \ref{prop:Polarization}, $\PPP_{\JJ}$ is a stable polynomial. $\PPP_{\JJ}$ is not symmetric and $\JJ$ is not exchangeable, thus its pmf $f_{\JJ}$ has not the maximal entropy in $\BBB_d(p)$.
	\end{example}

	\section{Dependence ordering in the class of $\mathcal{B}_d^{\Sigma}(\pp)$} \label{sect:WeakAssociationOrder}
	
	Since two Bernoulli random vectors $\II$ and $\II'$ in the same Fréchet class $\mathcal{B}_d(\pp)$ are not comparable in~$\unlhd$, we study the link between the strongly Rayleigh property and the dependence orders, $\preceq_{w-assoc}$ and $\preceq_{sm}$. 
	Proposition \ref{prop:fHminSM} and its Corollary \ref{corr:Conditioning} consider the simpler case of equal means, i.e. the class $\mathcal{B}_d^{\Sigma}(p)$.
	
	\begin{proposition} \label{prop:fHminSM}
		The maximum entropy pmf $f^H \in \mathcal{B}_d^{\Sigma}(p)$ that satisfies the strongly Rayleigh property is minimal under the supermodular order in the class of exchangeable Bernoulli pmfs with mean $p$.
	\end{proposition}
	\begin{proof}
		The maximum entropy pmf $f^H\in \mathcal{B}_d^{\Sigma}(p)$ is exchangeable. By  Proposition 5 in \cite{frostig2001comparison} it is minimal under the supermodular order in the class of exchangeable pmfs in $\mathcal{B}_d(p)$.
	\end{proof}
	The strongly Rayleigh property is closed under marginalization, but the $\Sigma$-countermonotonicity property is not. We provide an example of a Bernoulli random vector that satisfies the strongly Rayleigh property, that is not $\Sigma$-countermonotonic, and that is different from independence.
	The following Corollary \ref{corr:Conditioning} states that by marginalization of the $d$-dimensional maximum entropy pmf, $f^H_d\in \mathcal{B}_d^{\Sigma}(p)$, we obtain a $k$-dimensional pmf $f_{d|k}^H \in \mathcal{B}_k(p)$ that satisfies the strongly Rayleigh property but that is not $\Sigma$-countermonotonic and greater under the supermodular order than the $k$-dimensional maximum entropy pmf $f_k^H \in \mathcal{B}^{\Sigma}_k(p)$.
	\begin{corollary} \label{corr:Conditioning}
		Let $f_d^H$ be the maximum entropy pmf in the $d$-dimensional class $\mathcal{B}_d^{\Sigma}(p)$. Then, for $k \leq d$, the $k$-dimensional marginal pmf $f_{d|k}^H \in \mathcal{B}_k(p)$ of $f_d^H$ satisfies the strongly Rayleigh property and $f_k^H \preceq_{sm} f_{d|k}^H$, where $f^H_k \in \mathcal{B}^{\Sigma}_k(p)$ is the maximum entropy pmf in dimension $k$.
	\end{corollary}
	\begin{proof}
		If $pd \leq 1$ or $pd \geq d-1$, the result is trivial since the lower Fréchet bound is closed under marginalization. Otherwise,
		Lemma \ref{lemma:StabilityPreservers} implies that $f_{d|k}^H \in \mathcal{B}_k(p)$ satisfies the strongly Rayleigh property. 
		Given that $f_d^H$ is exchangeable, $f_{d|k}^H \in \mathcal{B}_k(p)$ is also exchangeable.
		By Proposition \ref{prop:fHminSM}, $f_k^H \preceq_{sm} f_{d|k}^H$.
	\end{proof}

	The following theorem proves that the class $\mathcal{B}_d^{\Sigma}(\pp)$ verifies a minimality property with respect to the supermodular order and, therefore, also with respect to the weak association order. 
	We recall that an antichain in a partially ordered set $(\mathcal{X}, \preceq_*)$ is a subset $A \subseteq \mathcal{X}$ such that any two elements in $A$ are not comparable according to $\preceq_*$. A chain is a subset $C \subseteq \mathcal{X}$ such that if $x,y\in C$ then $x\preceq_*y$ or $y\preceq_*x$ (see Chapter 12 in \cite{west2021combinatorial}).
	\begin{theorem} \label{thm:notorder}
		The following holds true:
		\begin{enumerate}
			\item\label{point1} The class $\mathcal{B}_d^{\Sigma}(\pp)$ is an antichain in the partially ordered set $(\mathcal{B}_d(\pp), \preceq_{sm})$.
			\item \label{point2} For any $f\in \mathcal{B}_d(\pp) \setminus \mathcal{B}_d^{\Sigma}(\pp)$, there exists either $f^{\Sigma}\in \mathcal{B}_d^{\Sigma}(\pp)$ such that $f^{\Sigma}\preceq_{sm}f$ or $f$ and $f^{\Sigma}$ are not comparable under the supermodular order.
		\end{enumerate}
	\end{theorem}
	\begin{proof}
		To prove Point \ref{point1}, let $f, f'\in \mathcal{B}_d^{\Sigma}(\pp)$ and $f\neq f'$.
		By Theorem 2.5.4 in \cite{muller2013duality}, $f\preceq_{sm}f'$ if and only if $f'$ can be obtained from $f$ by a finite number of supermodular transfers.
		Let $\ii, \kk \in \{0,1\}^d$, with $\sum_{j=1}^d i_j \leq \sum_{j=1}^d k_j$, and define $\ww = \ii \wedge \kk$ and $\zz = \ii \vee \kk$.  
		A simple supermodular transfer is a transfer given by $\eta (\frac{1}{2} \delta_{\ii} + \frac{1}{2} \delta_{\kk}) \to \eta (\frac{1}{2} \delta_{\ww} + \frac{1}{2} \delta_{\zz})$, where $\delta_{\xx}$ denotes the point mass in $\xx \in \{0,1\}^d$. If $f, f'\in \mathcal{B}_d^{\Sigma}(\pp)$, they have support on $\AAA^* =\mathcal{A}_m\cup\mathcal{A}_{m+1}$ (see Remark~\ref{remark:astar}). 
		Thus, since we must have $\ww,\zz \in \AAA^*$, and $\ww \preceq \zz$, it follows that $\ww = \ii$ and $\zz = \kk$. 
		Therefore, the only admissible supermodular transfer is the identity. Point \ref{point2} is a direct consequence of Theorem \ref{thm:EquivalentDef}. In fact, the class of $\Sigma$-countermonotonic pmfs is the class of $\Sigma_{cx}$-smallest pmfs defined by minimizing a class of supermodular functions.
	\end{proof}
	
	\begin{remark}
		Notice that Theorem \ref{thm:notorder} cannot be improved. In fact, there exists $f\in \mathcal{B}_d(\pp) \setminus \mathcal{B}_d^{\Sigma}(\pp)$ such that $f$ is not comparable with any $f^{\Sigma} \in \BBB_d^{\Sigma}(\pp)$.
		Indeed, the author of \cite{frostig2001comparison} proposes an algorithm to find sequences of random vectors decreasing in supermodular order, and in Example 3 of \cite{frostig2001comparison}, the algorithm stops at a Bernoulli pmf that is not $\Sigma_{cx}$-smallest. The proposed algorithm considers all supermodular transfers of Theorem 2.5.4 of \cite{muller2013duality}. 
		Thus, the algorithm stops at a pmf that is not comparable with any $f^{\Sigma} \in \BBB_d^{\Sigma}(\pp)$.
	\end{remark}
	
	From Theorem \ref{thm:notorder} it follows that the Bernoulli random vector $\JJ \in \mathcal{B}_d^{\Sigma}(\pp)$ in Proposition \ref{prop:Polarization}, that satisfies the strongly Rayleigh property, is not comparable in supermodular order with the maximum entropy $f^H \in \mathcal{B}_d^{\Sigma}(\pp)$.
	Using $\II^H$ and $\II^{\perp}$ in a given Fréchet class $\BBB_d(\pp)$, we can construct a parametric family of Bernoulli random vectors ordered in supermodular order, from a  $\Sigma$-countermonotonic pmf to independence.
	
	As we have seen in Example~\ref{ex:ctexample1}, since the $\Sigma$-countermonotonicity property does not imply negative association, the inequalities $\boldsymbol{I} \preceq_{w-assoc} \boldsymbol{I}^{\perp}$ and $\boldsymbol{I} \preceq_{sm} \boldsymbol{I}^{\perp}$ do not hold for all $\boldsymbol{I} \in \mathcal{B}_d^{\Sigma}(\pp)$. 
	Obviously, both $\boldsymbol{I}^{H} \preceq_{w-assoc} \boldsymbol{I}^{\perp}$ and $\boldsymbol{I}^{H} \preceq_{sm} \boldsymbol{I}^{\perp}$ hold since the strongly Rayleigh property implies negative association.

	\begin{proposition} \label{prop:FamilyNSDBernoulli}
		For $\alpha \in [0,1]$, let $\II^{(\alpha)}$ be a Bernoulli random vector with joint pmf $f^{(\alpha)}$ defined by 
		\begin{equation}
			f^{(\alpha)}(\ii)
			=
			(1-\alpha) f^H(\ii)
			+ 
			\alpha  f^{\perp}(\ii), 
			\quad \text{for }
			\ii \in \{0,1\}^d,
			\label{eq:FamilyNSDBernoulli}
		\end{equation}
		where $f^{\perp} \in \BBB_d(\pp)$ is the pmf of the Bernoulli random vector $\II^{\perp}$ with independent components and $f^H$ is the maximum entropy pmf of the class $\BBB_d^{\Sigma}(\pp)$.
		Then, the class $\{f^{(\alpha)} : \alpha \in [0,1] \}$ is a chain in the partially ordered set ($\mathcal{B}_d(\pp), \preceq_{sm}$).
		In particular, $f^{(\alpha)}$ satisfies the NSD property for every $\alpha \in [0,1]$, and, for every $0 \leq \alpha_1 \leq \alpha_2 \leq 1$, we have
		\begin{equation} \label{eq:BernoulliIncreasingInSM}
			\II^H \preceq_{sm} \II^{(\alpha_1)} \preceq_{sm} \II^{(\alpha_2)} \preceq_{sm} \II^{\perp},
		\end{equation}
		where $\II^H$ is a Bernoulli random vector with pmf $f^H$.
	\end{proposition}
	\begin{proof}
		From the definition of $f^{(\alpha)}$ in \eqref{eq:FamilyNSDBernoulli}, it follows that
		\begin{equation*}
			E[\phi(\II^{(\alpha)})] = (1 - \alpha) E[\phi(\II^H)] + \alpha E[\phi(\II^{\perp})]
		\end{equation*}
		for every real function $\phi \colon \{0,1\}^d \to \RR$, such that the expectations are finite.
		Given that $\II^H$ satisfies the strongly Rayleigh property, and also enjoys the NSD property, if $\phi$ is also supermodular, then $E[\phi(\II^H)] \leq E[\phi(\II^{\perp})]$.
		This implies that $E[\phi(\II^{(\alpha)})]$ increases as $\alpha$ increases. Therefore, we have \eqref{eq:BernoulliIncreasingInSM} and $f^{(\alpha)} \in NSD$ for every $\alpha \in [0,1]$.
	\end{proof}
	
	The random vector $\II^{(\alpha)}$ is not $\Sigma$-countermonotonic for $\alpha \in (0,1]$. The following counterexample shows that $\II^{(\alpha)}$ is not necessarily negatively associated; it does not, hence, always satisfy the strongly Rayleigh property.
	
	\begin{example}
		Let us consider the Fréchet class $\BBB_5(\tfrac{3}{10})$.
		The pmf $f^{\perp} \in \BBB_5(\tfrac{3}{10})$ of the vector $\II^{\perp}$ of independent random variables is  
		\begin{equation*}
			f^{\perp}(\ii) = \bigg( \frac{3}{10} \bigg)^{i_1+\dots+i_d} \bigg( \frac{7}{10} \bigg)^{d-(i_1+\dots+i_d)},
			\quad
			\text{for } \ii \in \{0,1\}^d,
		\end{equation*}
		while the maximum entropy pmf $f^H$ in the class $\BBB_5^{\Sigma}(\tfrac{3}{10})$ is the pmf of the Bernoulli exchangeable random vector $\II^H$ of the class $\BBB_5^{\Sigma}(\tfrac{3}{10})$.
		Since $p_{\bullet} = \frac{3}{2} \in (1,2)$, $f^H$ is given by:
		\begin{equation*}
			f^{H}(\ii) = 
			\begin{cases}
				\frac{1}{10}, &\text{for } \ii \in \mathcal{A}_{5,1}
				\\
				\frac{1}{20}, &\text{for } \ii \in \mathcal{A}_{5,2}
				\\
				0, &\text{elsewhere}
			\end{cases}.
		\end{equation*}
		Let us consider the random vector $\II^{(0.9)}$, whose pmf $f^{(0.9)}$ is defined by \eqref{eq:FamilyNSDBernoulli} with $\alpha = 0.9$, and consider the increasing functions $h_1(i_1,i_2) = e^{i_1+2i_2}$ and $h_2(i_3,i_4,i_5) = e^{3i_3+4i_4+6i_6}$.
		It is possible to check that $f^{(0.9)}$ is neither NA nor strongly Rayleigh, since
		\begin{equation*}
			Cov(h_1(I^{(0.9)}_1,I^{(0.9)}_2), h_2(I^{(0.9)}_3,I^{(0.9)}_4,I^{(0.9)}_5)) = 12.6715 > 0.
		\end{equation*}
	\end{example}

	\section{Acknowledgements}
	
	This work was partially supported by the Natural Sciences and Engineering Research Council of Canada (Cossette: 04273; Marceau: 05605). 
	P. Semeraro would like to thank \textit{\'Ecole d'actuariat, Université Laval,} for her delightful stay, during which part of this research was conducted.
	A. Mutti would like to thank \textit{\'Ecole d'actuariat, Université Laval} for the warm hospitality during his research stay.
	H. Cossette and E. Marceau would like to thank \textit{Dipartimento di Scienze Matematiche "G. L. Lagrange" (DISMA), Politecnico di Torino,} for their wonderful stay during which most of the paper was written. 
	
	\bibliographystyle{amsplain}
	\bibliography{ref}

\providecommand{\bysame}{\leavevmode\hbox to3em{\hrulefill}\thinspace}
\providecommand{\MR}{\relax\ifhmode\unskip\space\fi MR }
\providecommand{\MRhref}[2]{%
  \href{http://www.ams.org/mathscinet-getitem?mr=#1}{#2}
}
\providecommand{\href}[2]{#2}
\begin{thebibliography}{10}

\bibitem{bernard2017robust}
Carole Bernard, Ludger R{\"u}schendorf, Steven Vanduffel, and Jing Yao,
  \emph{How robust is the value-at-risk of credit risk portfolios?}, The
  European Journal of Finance \textbf{23} (2017), no.~6, 507--534.

\bibitem{borcea2010multivariate}
Julius Borcea and Petter Br{\"a}nd{\'e}n, \emph{Multivariate
  {P}{\'o}lya--{S}chur classification problems in the {W}eyl algebra},
  Proceedings of the London Mathematical Society \textbf{101} (2010), no.~1,
  73--104.

\bibitem{borcea2009negative}
Julius Borcea, Petter Br{\"a}nd{\'e}n, and Thomas Liggett, \emph{Negative
  dependence and the geometry of polynomials}, Journal of the American
  Mathematical Society \textbf{22} (2009), no.~2, 521--567.

\bibitem{boutsikas2000bound}
Michael~V. Boutsikas and Markos~V. Koutras, \emph{A bound for the distribution
  of the sum of discrete associated or negatively associated random variables},
  The Annals of Applied Probability \textbf{10} (2000), no.~4, 1137--1150.

\bibitem{branden2007polynomials}
Petter Br{\"a}nd{\'e}n, \emph{Polynomials with the half-plane property and
  matroid theory}, Advances in Mathematics \textbf{216} (2007), no.~1,
  302--320.

\bibitem{branden2012negative}
Petter Br{\"a}nd{\'e}n and Johan Jonasson, \emph{Negative dependence in
  sampling}, Scandinavian Journal of Statistics \textbf{39} (2012), no.~4,
  830--838.

\bibitem{brown1986fundamentals}
Lawrence~D Brown, \emph{Fundamentals of {S}tatistical {E}xponential {F}amilies:
  with {A}pplications in {S}tatistical {D}ecision {T}heory}, IMS, 1986.

\bibitem{Bauerle2006stochastic}
Nicole Bäuerle and Alfred Müller, \emph{Stochastic orders and risk measures:
  Consistency and bounds}, Insurance: Mathematics and Economics \textbf{38}
  (2006), no.~1, 132--148.

\bibitem{chen2000general}
Sean~X Chen, \emph{General properties and estimation of conditional {B}ernoulli
  models}, Journal of Multivariate Analysis \textbf{74} (2000), no.~1, 69--87.

\bibitem{chen1994weighted}
Xiang-Hui Chen, Arthur~P Dempster, and Jun~S Liu, \emph{Weighted finite
  population sampling to maximize entropy}, Biometrika \textbf{81} (1994),
  no.~3, 457--469.

\bibitem{cheung2013characterization}
Ka~Chun Cheung and Ambrose Lo, \emph{Characterizations of counter-monotonicity
  and upper comonotonicity by (tail) convex order}, Insurance: Mathematics and
  Economics \textbf{53} (2013), no.~2, 334--342.

\bibitem{christofides2004connection}
Tasos~C Christofides and Eutichia Vaggelatou, \emph{A connection between
  supermodular ordering and positive/negative association}, Journal of
  Multivariate Analysis \textbf{88} (2004), no.~1, 138--151.

\bibitem{dall2012advances}
G.~Dall'Aglio, S.~Kotz, and G.~Salinetti, \emph{Advances in {P}robability
  {D}istributions with {G}iven {M}arginals: {B}eyond the {C}opulas}, vol.~67,
  Springer Science \& Business Media, 2012.

\bibitem{dall1972frechet}
G~Dall’Aglio, \emph{Fr{\'e}chet classes and compatibility of distribution
  functions}, Symposia Mathematica, vol.~9, 1972, pp.~131--150.

\bibitem{de1997computational}
Mark De~Berg, Marc Van~Kreveld, Mark Overmars, and Otfried Schwarzkopf,
  \emph{Computational {G}eometry}, Springer, 1997.

\bibitem{dhaene1999safest}
Jan Dhaene and Michel Denuit, \emph{The safest dependence structure among
  risks}, Insurance: Mathematics and Economics \textbf{25} (1999), no.~1,
  11--21.

\bibitem{duminil2016quantitative}
Hugo Duminil-Copin, Dmitry Ioffe, and Yvan Velenik, \emph{A quantitative
  {B}urton--{K}eane estimate under strong {FKG} condition}, The Annals of
  Probability \textbf{44} (2016), no.~5, 3335--3356.

\bibitem{fontana2021model}
Roberto Fontana, Elisa Luciano, and Patrizia Semeraro, \emph{Model risk in
  credit risk}, Mathematical Finance \textbf{31} (2021), no.~1, 176--202.

\bibitem{fontana2018representation}
Roberto Fontana and Patrizia Semeraro, \emph{Representation of multivariate
  {B}ernoulli distributions with a given set of specified moments}, Journal of
  Multivariate Analysis \textbf{168} (2018), 290--303.

\bibitem{fontana2024bernoulli}
\bysame, \emph{The {B}ernoulli structure of discrete distributions}, arXiv
  preprint, arXiv:2410.13920 (2024).

\bibitem{fontana2024high}
\bysame, \emph{High dimensional {B}ernoulli distributions: Algebraic
  representation and applications}, Bernoulli \textbf{30} (2024), no.~1,
  825--850.

\bibitem{fortuin1971correlation}
Cees~M. Fortuin, Pieter~W. Kasteleyn, and Jean Ginibre, \emph{Correlation
  inequalities on some partially ordered sets}, Communications in Mathematical
  Physics \textbf{22} (1971), 89--103.

\bibitem{frostig2001comparison}
Esther Frostig, \emph{Comparison of portfolios which depend on multivariate
  {B}ernoulli random variables with fixed marginals}, Insurance: Mathematics
  and Economics \textbf{29} (2001), no.~3, 319--331.

\bibitem{gerber2019negative}
Mathieu Gerber, Nicolas Chopin, and Nick Whiteley, \emph{Negative association,
  ordering and convergence of resampling methods}, The Annals of Statistics
  \textbf{47} (2019), no.~4, 2236--2260.

\bibitem{ghosh2017multivariate}
Subhroshekhar Ghosh, Thomas~M Liggett, and Robin Pemantle, \emph{Multivariate
  {CLT} follows from strong {R}ayleigh property}, 2017 Proceedings of the
  Fourteenth Workshop on Analytic Algorithmics and Combinatorics (ANALCO),
  SIAM, 2017, pp.~139--147.

\bibitem{hermon2023modified}
Jonathan Hermon and Justin Salez, \emph{Modified log-{S}obolev inequalities for
  strong-{R}ayleigh measures}, The Annals of Applied Probability \textbf{33}
  (2023), no.~2, 1501--1514.

\bibitem{hu2000negatively}
Taizhong Hu, \emph{Negatively superadditive dependence of random variables with
  applications}, Chinese Journal of Applied Probabability and Statistics
  \textbf{16} (2000), no.~2, 133--144.

\bibitem{joag1983negative}
Kumar Joag-Dev and Frank Proschan, \emph{Negative association of random
  variables with applications}, The Annals of Statistics (1983), 286--295.

\bibitem{joe1997multivariate}
Harry Joe, \emph{Multivariate {M}odels and {M}ultivariate {D}ependence
  {C}oncepts}, CRC press, 1997.

\bibitem{karlin1980classes}
Samuel Karlin and Yosef Rinott, \emph{Classes of orderings of measures and
  related correlation inequalities ii. multivariate reverse rule
  distributions}, Journal of Multivariate Analysis \textbf{10} (1980), no.~4,
  499--516.

\bibitem{kimeldorf1978monotone}
George Kimeldorf and Allan~R Sampson, \emph{Monotone dependence}, The Annals of
  Statistics (1978), 895--903.

\bibitem{kimeldorf1989framework}
\bysame, \emph{A framework for positive dependence}, Annals of the Institute of
  Statistical Mathematics \textbf{41} (1989), no.~1, 31--45.

\bibitem{lauzier2023pairwise}
Jean-Gabriel Lauzier, Liyuan Lin, and Ruodu Wang, \emph{Pairwise
  counter-monotonicity}, Insurance: Mathematics and Economics \textbf{111}
  (2023), 279--287.

\bibitem{lee2014multidimensional}
Woojoo Lee and Jae~Youn Ahn, \emph{On the multidimensional extension of
  countermonotonicity and its applications}, Insurance: Mathematics and
  Economics \textbf{56} (2014), 68--79.

\bibitem{muller2013duality}
Alfred M{\"u}ller, \emph{Duality theory and transfers for stochastic order
  relations}, Stochastic Orders in Reliability and Risk: In Honor of Professor
  Moshe Shaked (2013), 41--57.

\bibitem{pemantle2000towards}
Robin Pemantle, \emph{Towards a theory of negative dependence}, Journal of
  Mathematical Physics \textbf{41} (2000), no.~3, 1371--1390.

\bibitem{puccetti2012computation}
Giovanni Puccetti and Ludger R{\"u}schendorf, \emph{Computation of sharp bounds
  on the distribution of a function of dependent risks}, Journal of
  Computational and Applied Mathematics \textbf{236} (2012), no.~7, 1833--1840.

\bibitem{puccetti2015extremal}
Giovanni Puccetti and Ruodu Wang, \emph{Extremal dependence concepts},
  Statistical Science \textbf{30} (2015), no.~4, 485--517.

\bibitem{shaked2007stochastic}
Moshe Shaked and J~George Shanthikumar, \emph{Stochastic {O}rders}, Springer,
  2007.

\bibitem{tchen1980inequalities}
Andr{\'e}~H Tchen, \emph{Inequalities for distributions with given marginals},
  The Annals of Probability (1980), 814--827.

\bibitem{tille2006sampling}
Yves Till{\'e}, \emph{Sampling {A}lgorithms}, Springer, 2006.

\bibitem{wagner2011multivariate}
David Wagner, \emph{Multivariate stable polynomials: theory and applications},
  Bulletin of the American Mathematical Society \textbf{48} (2011), no.~1,
  53--84.

\bibitem{west2021combinatorial}
Douglas~B West, \emph{Combinatorial mathematics}, Cambridge University Press,
  2021.

\end{thebibliography}

\end{document}